\pdfoutput=1
\RequirePackage{ifpdf}
\ifpdf 
\documentclass[pdftex]{sigma}
\else
\documentclass{sigma}
\fi

\numberwithin{equation}{section}

\newtheorem{Theorem}{Theorem}[section]
\newtheorem{Conjecture}[Theorem]{Conjecture}
\newtheorem{Lemma}[Theorem]{Lemma}
 { \theoremstyle{definition}
\newtheorem{Example}[Theorem]{Example}
\newtheorem{Remark}[Theorem]{Remark} }

\usepackage{palatino}
\usepackage{mathrsfs}
\usepackage{euscript}
\input xy \xyoption {all}

\usepackage{tikz}
\usetikzlibrary{positioning}
\usetikzlibrary{calc}
\usetikzlibrary{decorations.pathreplacing}

\RequirePackage{color}
\definecolor{myred}{rgb}{0.75,0,0}
\definecolor{mygreen}{rgb}{0,0.5,0}
\definecolor{myblue}{rgb}{0,0,0.65}

\usepackage{cleveref}

 \RequirePackage{ae, aecompl, aeguill} 

 \def\hg{{\mathfrak h}} 
 \def\ig{{\mathfrak i}}

  \def\QM{{\mathbb{Q}}}
  \def\RM{{\mathbb{R}}}

  \def\ZM{{\mathbb{Z}}}

\def\AC{{\mathcal{A}}}
\def\HC{{\mathcal{H}}}
\def\MC{{\mathcal{M}}}
\def\RC{{\mathcal{R}}}
\def\AS{{\EuScript A}}

\def\a{\alpha}

\def\g{\gamma}

\def\l{\lambda}

\newcommand{\nc}{\newcommand} \newcommand{\renc}{\renewcommand}

\newcommand{\rdots}{\mathinner{ \mkern1mu\raise1pt\hbox{.}
 \mkern2mu\raise4pt\hbox{.}
 \mkern2mu\raise7pt\vbox{\kern7pt\hbox{.}}\mkern1mu}}

\def\un{\underline}

\def\p{{}^p}

\def\to{\rightarrow}

\def\longto{\longrightarrow}

\nc{\triright}{\stackrel{[1]}{\to}}
\nc{\longtriright}{\stackrel{[1]}{\longto}}

\nc{\Hb}{H^\bullet}

\nc{\Br}{\mathcal{B}}
\nc{\HotRR}{{}_R\mathcal{K}_R}
\nc{\HotR}{\mathcal{K}_R}
\nc{\excise}[1]{}
\nc{\defect}{\text{df}}
\nc{\h}[1]{\underline{H}_{#1}}

\nc{\Ga}{\mathbb{G}_a} 
\nc{\Gm}{\mathbb{G}_m} 

\nc{\Perv}{{\mathbf{P}}}

\nc{\IH}{{\mathrm{IH}}}

\nc{\ic}{\mathbf{IC}}

\nc{\gl}{{\mathfrak{gl}}}
\renc{\sl}{{\mathfrak{sl}}}
\renc{\sp}{{\mathfrak{sp}}}

\renc{\Im}{\textrm{Im}}

\nc{\HBM}{H^{BM}}

\DeclareMathOperator{\id}{id}

\def\fW{{}^fW}

\nc{\simto}{\stackrel{\sim}{\to}}

\newlength{\classpageheight} 
\newlength{\classpagewidth} 

\def\SL{\textrm{SL}}
\def\GL{\textrm{GL}}
\def\Chi{\mathscr{X}}
\def\int{\textrm{int}}

\def\curlabel{\textrm{CurrentLabel}}
\def\curdir{\textrm{CurrentDirection}}
\def\curpoint{\textrm{CurrentPoint}}
\def\cursign{\textrm{CurrentSign}}

\def\start{\textrm{start}}
\def\new{\textrm{new}}
\def\wall{\textrm{wall}}
\def\Wf{{W_f}}

\nc{\He}{\mathrm{H}} 
\nc{\Hee}{\mathrm{H}^\mathrm{ext}} 
\nc{\Zvv}{\mathbb{Z}[v]}
\nc{\Zv}{\mathbb{Z}\big[v^{\pm 1}\big]}
\nc{\AntiS}{\mathrm{AS}}
\nc{\ASCat}{\mathcal{AS}}
\def\sgn{{\mathrm{sgn}}}
\def\f{{}^f}
\def\p{{}^p}
\nc{\Wa}{\mathcal{W}}
\nc{\Sa}{\mathcal{S}}
\nc{\Sf}{S_f}
\nc{\pdot}{ \bullet_p}

\begin{document}
\allowdisplaybreaks

\newcommand{\arXivNumber}{1703.05898}

\renewcommand{\thefootnote}{}

\renewcommand{\PaperNumber}{015}

\FirstPageHeading

\ShortArticleName{Billiards and Tilting Characters for ${\rm SL}_3$}

\ArticleName{Billiards and Tilting Characters for $\boldsymbol{{\rm SL}_3}$\footnote{This paper is a~contribution to the Special Issue on the Representation Theory of the Symmetric Groups and Related Topics. The full collection is available at \href{https://www.emis.de/journals/SIGMA/symmetric-groups-2018.html}{https://www.emis.de/journals/SIGMA/symmetric-groups-2018.html}}}

\Author{George LUSZTIG~$^\dag$ and Geordie WILLIAMSON~$^\ddag$}

\AuthorNameForHeading{G.~Lusztig and G.~Williamson}

\Address{$^\dag$~Massachusetts Institute of Technology, Cambridge, MA, USA}
\EmailD{\href{mailto:gyuri@math.mit.edu}{gyuri@math.mit.edu}}
\URLaddressD{\url{http://www-math.mit.edu/~gyuri/}}

\Address{$^\ddag$~Sydney University, Sydney, NSW, Australia}
\EmailD{\href{mailto:g.williamson@sydney.edu.au}{g.williamson@sydney.edu.au}}
\URLaddressD{\url{http://www.maths.usyd.edu.au/u/geordie/}}

\ArticleDates{Received July 18, 2017, in f\/inal form February 16, 2018; Published online February 21, 2018}

\Abstract{We formulate a conjecture for the second generation characters of indecomposable tilting modules for~${\rm SL}_3$. This gives many new conjectural decomposition numbers for symmetric groups. Our conjecture can be interpreted as saying that these characters are governed by a discrete dynamical system (``billiards bouncing in alcoves''). The conjecture implies that decomposition numbers for symmetric groups display (at least)
exponential growth.}


\Classification{20C20; 17B10; 20C30}

\renewcommand{\thefootnote}{\arabic{footnote}}
\setcounter{footnote}{0}

\section{Introduction}

We formulate a conjecture for the second generation characters of indecomposable tilting mo\-du\-les for $\SL_3$ in characteristic $p > 2$. These conjectures resulted from our attempts to understand data~\cite{WData} obtained following a 9 month calculation in magma on a supercomputer at the MPIM in Bonn. These results go far beyond existing calculations and are obtained using a new algorithm~\cite{WSym}. The algorithm relies in an essential way on ideas of Libedinsky, Riche and the second author (see~\cite{LiW, RW}). The behaviour we observe appears highly non-trivial, which suggests that proving anything might be dif\/f\/icult.

On the next page the reader will f\/ind a picture. This picture was obtained by analyzing the output of computer calculations for $p = 5$. Exactly how this picture is used to produce (second generation) tilting characters will be explained in the f\/inal section. Before reading the rest of the paper
the reader is invited to consider this picture and try to discern any patterns. This paper is an attempt to explain this picture, as well as similar pictures for $p = 3$ and $7$.

\begin{figure}\centering
 \includegraphics[angle=90,width=135mm]{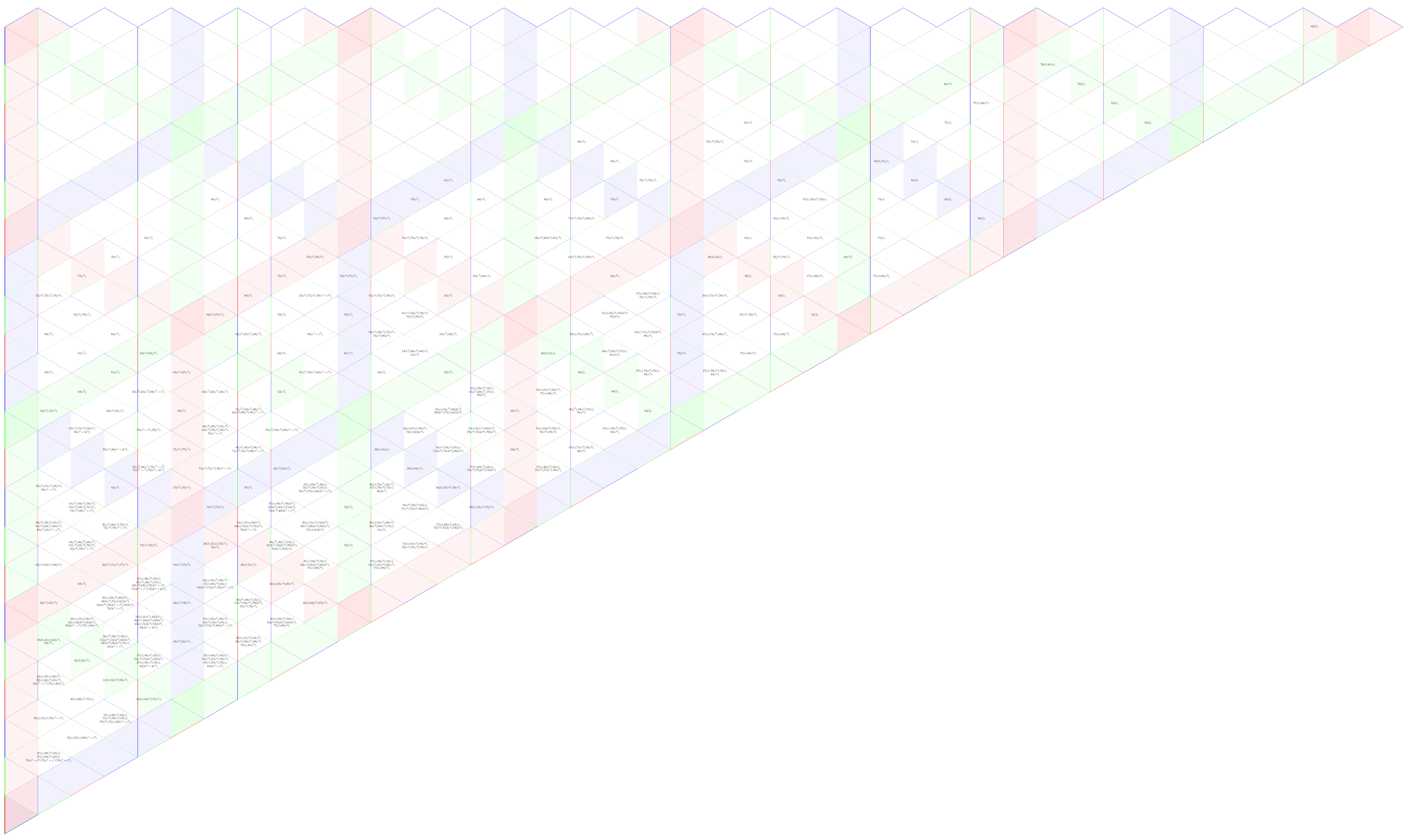}
\caption{Second generation pattern for $p = 5$ up to $i = 81$.} \label{fig:p5}
\end{figure}

\section{Generational philosophy}

Let $G$ denote a split simple and simply connected algebraic group over a f\/ield $\Bbbk$ of characteristic~$p$. We f\/ix a Borel subgroup and maximal torus $T \subset B \subset G$. We will try to follow the notation of~\cite{WTakagi}. In particular:
\begin{center}
\begin{tabular}{rl}
$\Chi$, $\Chi_+$: & weights, dominant weights; \\
$\Phi$, $\Phi_+$: & roots, positive roots; \\
$\Phi^\vee$, $\Phi_+^\vee$: &coroots, positive coroots; \\
$\Sigma$, $\Sigma^\vee$: & simple roots, simple coroots; \\
$\Chi_{<p}$: & $p$-restricted weights; \\
$\rho$, $\a_0^\vee$: & half-sum of $\Phi_+$, highest short coroot; \\
$\Wf$, $\Wa$: & f\/inite Weyl group, af\/f\/ine Weyl group; \\
$\Sf$, $\Sa$: & f\/inite simple ref\/lections, af\/f\/ine simple ref\/lections; \\
$\f\Wa$: & minimal coset representatives for $\Wf \setminus \Wa$; \\
$\pdot$, $h$: & $p$-dilated dot action of $\Wa$ on $\Chi$, Coxeter
 number.
\end{tabular}
\end{center}
For every dominant weight $\l \in \Chi_+$, let $L_\l$ (resp.~$\Delta_\l$, resp.~$T_\l$) denote the simple (resp. standard, resp.\ indecomposable tilting) module with highest weight~$\l$.

\subsection{Generations for simple characters} \label{sec:simple}

For f\/ixed $p$ and any highest weight $\l \in \Chi_+$ the f\/irst author has recently def\/ined characters \cite{LuGenerations}
\begin{gather*}
\chi_\l^0, \chi_\l^1, \chi_\l^2, \dots, \chi_\l^\infty \in (\ZM \Chi)^\Wf
\end{gather*}
with the following properties:
\begin{enumerate}\itemsep=0pt
\item[1)] $\chi_\l^0$ is the character of the simple highest weight module in characteristic $0$ (i.e., $\chi_\l^0$ is given by Weyl's character formula);
\item[2)] $\chi_\l^1$ is the character of the simple highest weight module for the quantum group at a $p^{\rm th}$-root of unity;
\item[3)] $\chi_\l^n$ is obtained from $\chi_\l^{n-1}$ by a formula involving Kazhdan--Lusztig polynomials;
\item[4)] if $\l = \l_0 + p\l_1 + \dots + p^n\l_n$ with $\l_i \in \Chi_{<p}$ and $\langle \alpha_0^\vee, \l_n + \rho \rangle \le p$ (a~generalised Jantzen condition) then $\chi_\l^n = \chi_\l^{n+1} = \dots = \chi_\l^{\infty}$;
\item[5)] if $p$ is large then $\chi_\l^\infty$ is the character of $L_\l$.
\end{enumerate}

\begin{Remark} \label{rem:generations} Recall that the Steinberg tensor product theorem has ``one step'' for the quantum group at a root of unity, and ``inf\/initely many steps'' for~$G$. The characters $\chi_\l^n$ can be thought of as the simple characters for an object with an ``$n$ step'' Steinberg tensor product theorem. It is not known whether such an object exists. For $\mathfrak{sl}_2$ such an object (for any $n$) has recently been proposed by Angiono~\cite{Angiono}. It seems likely that one can combine recent work of Elias \cite{EQS} with work of Riche and the second author~\cite{RW} to construct such an object in type $A$ for $n = 2$.
\end{Remark}

\subsection{Generations for tilting characters} As for simple characters, there should exist approximations to tilting characters. That is, for any $p$ and dominant weight $\l \in \Chi$ there should exist characters
\begin{gather*}
\theta_\l^0, \theta_\l^1, \theta_\l^2, \dots, \theta_\l^\infty \in (\ZM \Chi)^\Wf
\end{gather*}
with the following properties:
\begin{enumerate}\itemsep=0pt
\item[1)] $\theta_\l^0$ is the character of the simple highest weight module in characteristic $0$ (i.e., $\theta_\l^0$ is given by Weyl's character formula);
\item[2)] $\theta_\l^1$ is the character of the indecomposable tilting module for the quantum group at a~$p^{\rm th}$-root of unity;
\item[3)] $\theta_\l^n$ is a positive linear combination of the $\theta_\mu^{n-1}$'s for all $n \ge 1$;
\item[4)] if $\langle \a_0^\vee, \l + \rho \rangle \le p^{n+1}$ then $\theta_\l^n = \theta_\l^{n+1} = \dots = \theta_\l^{\infty}$;
\item[5)] if $p$ is large then $\theta_\l^{\infty}$ is the character of $T_\l$.
\end{enumerate}

\begin{Remark} Parts (4) and (5) for $n = 1$ imply Andersen's conjecture \cite{AFilt1}, which is still open (even for $p$ large).
\end{Remark}

\begin{Remark}The characters $\theta_\l^n$ are def\/ined for a subset of $\Chi_+$ (roughly, those for which one can apply the tilting tensor product
theorem) in~\cite{LusztigWilliamson}.
\end{Remark}

\subsection[Our conjecture at $v = 1$]{Our conjecture at $\boldsymbol{v = 1}$} \label{sec:v1}
Let us assume that $p \ge h$, and consider the $p$-dilated action of the af\/f\/ine Weyl group from above. For any choice of ``generation parameters'' $0 \le m \le n \le \infty$ we can write
\begin{gather} \label{eq:2}
 \theta^n_{x \pdot 0} = \sum_{y \in W} d_{y,x}^{m,n} \cdot \theta_{y \pdot 0}^m
\end{gather}
for some $d_{y,x}^{m,n} \in \ZM_{\ge 0}$. (The fact that we can form such expressions follows from the linkage principle and our assumptions above.) The question of determining characters of indecomposable tilting characters is equivalent to determining the coef\/f\/icients $d_{y,x}^{\infty, 0}$ for all $y, x \in \fW$.

\begin{Remark} A philosophy underlying the current paper is that it might be easier to calculate the $d_{y,x}^{n+1,n}$ for all $n$, rather than calculate the $d_{y,x}^{\infty, 0}$ directly. This is the case for simple characters, as explained in Section~\ref{sec:simple}.
\end{Remark}

The following is known about these coef\/f\/icients:
\begin{enumerate}\itemsep=0pt
\item[1)] when $G = \GL_n$, the $d_{y,x}^{\infty,0}$ are equal to decomposition numbers for symmetric groups, by work of Donkin \cite{DonkinTilting} and Erdmann~\cite{Erdmann};
\item[2)] the $d_{y,x}^{1,0}$ are given as the value at 1 of certain parabolic Kazhdan--Lusztig polynomials, by work of Soergel \cite{SoergelKippKM, SoeKL};
\item[3)] the $d_{y,x}^{\infty, 1}$ are equal to coef\/f\/icients of James' ``adjustment matrix''~\cite{James}.
\end{enumerate}

\begin{Remark} Aside from the above, very little is known. The papers \cite{JensenTilting, Parker} contain some interesting calculations for~$\SL_3$
\end{Remark}

In this paper we formulate a conjecture for ($v$-analogues of the) coef\/f\/icients $d_{y,x}^{2,1}$, when \smash{$G = \SL_3$}. (Note that in the case of $G = \SL_3$ the coef\/f\/icients $d_{y,x}^{1,0}$ given by Soergel's algorithm are easily calculated, and may be described by closed formulas.) By property (4) of the previous section, this provides a conjecture for many $d_{y,x}^{\infty,1}$, and hence for many coef\/f\/icients of the adjustment matrix (and hence decomposition numbers) for three row partitions.

\subsection[Generations for the $p$-canonical basis]{Generations for the $\boldsymbol{p}$-canonical basis} \label{sec:genpcan}
Let $\He$ denote the Iwahori--Hecke algebra of the af\/f\/ine Weyl group $\Wa$ over $\Zv$. Consider the anti-spherical module (see, e.g., \cite{RW,WTakagi}):
\begin{gather*}
\AntiS_v := \sgn_v \otimes_{\He_f} \He = \bigoplus_{x \in \f\Wa} \ZM\big[v^{\pm 1}\big] n_x.
\end{gather*}
The anti-spherical module has a canonical basis $\{ \un{n}_x \, | \, x \in \f\Wa \}$ and a $p$-canonical basis $\{ \p\un{n}_x \, | \, x \in \f\Wa \}$ \cite{JensenW,RW}.

In \cite{RW} (see also \cite[Section~2.8]{WTakagi}) it is conjectured that if we write
\begin{gather*}
\p\un{n}_x = \sum \p{m_{y,x}} n_y,
\end{gather*}
then the values at 1 of the coef\/f\/icients $\p{m_{y,x}}$ express the characters of tilting modules in terms of Weyl characters.\footnote{See \cite[Conjecture~1.7]{RW} for a precise formulation. Note that this conjecture is expected to hold without any restrictions on~$p$.} Let
us give a precise version of this conjecture for $p \ge h$: for all $x, y \in \f\Wa$ we have (in the notation of Section~\ref{sec:v1})
\begin{gather} \label{eq:tilting-mult}
d_{y,x}^{\infty,0} = \p{m_{y,x}}(1).
\end{gather}
(Recall that, by the translation principle and still under our assumption $p \ge h$, knowledge of the left hand side of~\eqref{eq:tilting-mult} for all $x, y \in \f\Wa$ implies knowledge of the characters of all tilting modules for~$G$.)

\begin{Remark} In \cite{RW} this conjecture is proved for $\GL_n$ and $p > h$. In \cite{ELo} this conjecture is proved for $\GL_n$ for general $p$. In \cite{AMRW1, AMRW2, ARFM} this conjecture is proved in all types for $p > h$.
\end{Remark}

The presence of $v$ is the shadow of a non-trivial grading on the category of tilting modules.\footnote{This grading is conjectural in general. In
\cite{AMRW1,AMRW2,ARFM,ELo, RW} its existence is established in many cases.} It is natural to expect that the above approximations to tilting characters can also be made to respect this grading. That is, we expect that for all $p > h$ and $x \in \f\Wa$, there exist elements
\begin{gather*}
\p\un{n}_x^1, \p\un{n}_x^2, \dots, \p\un{n}_x^{\infty} \in \AntiS_v
\end{gather*}
such that:
\begin{enumerate}\itemsep=0pt
\item[1)] $\p\un{n}_x^1 = \un{n}_x$;
\item[2)] $\p\un{n}_x^n$ is a $\ZM_{\ge 0}[v^{\pm 1}]$-linear combination of the $\p\un{n}_y^{n-1}$'s for all $n \ge 1$;
\item[3)] if $\langle \a_0^\vee, x \pdot 0 + \rho \rangle \le p^{n+1}$ then
$\p\un{n}_x^n = \p\un{n}_x^{n+1} = \dots = \p\un{n}_x^\infty$;
\item[4)] if $p$ is large then $\p\un{n}_x^\infty = \p\un{n}_x$.
\end{enumerate}

\section{The new algorithm}

In this section we attempt to give the reader some idea of how we perform the calculations which led to our conjecture. The second author hopes to give more details in~\cite{WSym}. As explained in the previous section, the main theorems of \cite{ELo, RW} (see also \cite{AMRW1, AMRW2}) imply that it is enough to calculate the $p$-canonical basis $\{ \p \un{n}_x \, | \, x \in \fW \}$ in the anti-spherical module $\AntiS_v$.

Let $\HC$ denote the diagrammatic Hecke category associated to the af\/f\/ine Cartan matrix of type $\widetilde{A}_n$. Recall that $\HC$ is built starting from a ``realisation''
\begin{gather*}
\big( \{ \alpha_s^\vee \}_{s \in S} \subset \hg, \{ \alpha_s \}_{s \in S} \subset \hg^*\big),
\end{gather*}
where $\hg$ is a free and f\/inite rank $\ZM$-module, $\hg^*$ is its dual, and the usual formulas def\/ine a~representation of $W$. It will be important below that our realisation is chosen so that the simple roots $\{ \alpha_s \}_{s \in S} \subset \hg^*$ are linearly independent.

The Hecke category is def\/ined over $\ZM$. After f\/ixing a prime $p$, extension of scalars yields the Hecke category $\HC_{\ZM_p}$ def\/ined over the $p$-adic integers. Let $\ASCat_{\ZM_p}$ denote the anti-spherical category over $\ZM_p$ considered in \cite{LiW, RW}. It is a right $\HC_{\ZM_p}$-module category and one has a canonical identif\/iciation of right $[\HC_{\ZM_p}]_\oplus = H$-modules:
\begin{gather*}
\AntiS_v \simto [\ASCat_{\ZM_p}]_{\oplus},
\end{gather*}
where $[-]_{\oplus}$ denotes the split Grothendieck group of an additive category (see \cite{LiW, RW} for more details). Under this isomorphism the classes of the indecomposable self-dual objects yields the $p$-canonical basis in $\AntiS_v$.

The anti-spherical category $\ASCat_{\ZM_p}$ is def\/ined by explicit generators and relations. Thus the question of determining the characters of its indecomposable objects (and hence the $p$-canonical basis) is a~concrete question of f\/inding idempotents in certain f\/inite rank $\ZM_p$-algebras. However these diagrammatic calculations are prohibitively dif\/f\/icult in all but the simplest cases.

Instead we exploit the philosophy of localisation. To $\ASCat_{\ZM_p}$ are associated progressively simpler categories:
\begin{gather*}
\ASCat_{\ZM_p} \rightsquigarrow \ASCat_{\QM_p} \rightsquigarrow Q \otimes_R \ASCat_{\QM_p},
\end{gather*}
where $\rightsquigarrow$ denotes some form of localisation. The f\/irst localisation $\ASCat_{\QM_p}$ is obtained from $\ASCat_{\ZM_p}$ by inverting $p$. The second is the main object of study of~\cite{LiW}: the polynomial ring $R = \QM_p[\widetilde{\alpha}]$ ($\widetilde{\alpha}$~denotes the af\/f\/ine simple root) acts on the left on all hom spaces in $\ASCat_{\QM_p}$, and after inverting~$\widetilde{\alpha}$ one obtains $Q \otimes_R
\ASCat_{\QM_p}$, where $Q$ denotes $\QM_p(\widetilde{\alpha})$. Somewhat surprisingly, this category is semi-simple~\cite{LiW}. (This is analogous to deformed category $\mathcal{O}$, which is semi-simple for generic parameters. A big dif\/ference in the current setting is that the deformation ring $R$ is always one dimensional.)

The algorithm we used to calculate the $p$-canonical basis now proceeds in two steps:
\begin{enumerate}\itemsep=0pt
\item Firstly, $\ASCat_{\QM_p}$ is described as a quiver with relations, and the action of the Hecke category~$\HC_{\ZM_p}$ on~$\ASCat_{\QM_p}$ is
described explicitly, in terms of the quiver. This is already a very non-trivial task, and is only feasible for $\SL_2$, $\SL_3$ and perhaps~$\SL_4$. It is possible in these cases thanks to the localisation $Q \otimes_R \ASCat_{\QM_p}$ (where any calculation can be reduced to a~calculation in matrices with entries in~$Q$), the fact that the Kazhdan--Lusztig conjectures hold in~$\AS_{\QM_p}$ (thus one has graded dimensions for hom spaces and knows how the generators of~$\HC_{\ZM_p}$ act on the Grothendieck group etc.), and the fact that the Kazhdan--Lusztig theo\-ry of these anti-spherical modules in low rank is not too complicated (Kazhdan--Lusztig polynomials can be written down explicitly).
\item Secondly, via the above localisations one can describe~$\ASCat_{\ZM_p}$ as a~$\ZM_p$-lattice inside~$\ASCat_{\QM_p}$. Moreover, this lattice is the smallest lattice which contains the generating object and is stable under the action of the Hecke category $\HC_{\ZM_p}$. One may describe this lattice explicitly and inductively via the $\HC_{\ZM_p}$-action, using the philosophy of the light leaf basis. Details of how this is done in this setting will be contained in \cite{WSym}.
\end{enumerate}

\begin{Remark} Recently, L.T.~Jensen has done calculations which omit the f\/irst localisation and calculate the $p$-canonical basis using only the localisation $Q \otimes_R \ASCat_{\QM_p}$. This is a much simpler approach, and appears (much to the second author's surprise!) not to be any slower than the approach described above. It may well be that Jensen's modif\/ication of the algorithm ends up being the more ef\/fective.
\end{Remark}

\section{Billiards}

For the rest of the paper we f\/ix $G = \SL_3$, $\Chi = \ZM\varpi_1 \oplus \ZM \varpi_2$, $\Sigma = \{ \alpha_1, \alpha_2 \}$ etc.

Fix $\ell \ge 1$. (For applications to representation theory $\ell$ will be prime, however for the moment~$\ell$ can be any positive integer. Soon we will assume $\ell \ge 3$.) Elements of the set
\begin{gather*}
 \MC := \ZM_{\ge 0} \times \big\{ v^k \, | \, k \in \ZM \big\}
\end{gather*}
will be called \emph{labels}. Labels will often be denoted by an integer followed by a bracketed power of $v$; e.g.,~$51(v^7)$.

 A \emph{labelled point} is an element of $\Chi \times \MC$. Labelled points will usually be denoted $(\mu,m)$ with $\mu \in \Chi$ and $m \in \MC$ or $(\mu,n,v^k)$ with $\mu \in \Chi$ and $n \in \ZM_{\ge 0}$.

Our goal in this section is to describe an algorithm which produces a~multiset (i.e., set with multiplicities) of labelled points via an inductive procedure. Throughout this section we work entirely with multisets. All operations (union, dif\/ference, \dots) are to be understood in the context
of multisets.

\subsection{Overview of the algorithm} 
We view the dominant weights $\Chi_+$ as the vertices of a directed graph $\Gamma$ with edges $\l \to \l + \g$ if $\l, \l + \g \in \Chi_+$ and $\g \in \{\varpi_1, \varpi_2 - \varpi_1, -\varpi_2 \}$:
\begin{gather*}
\Gamma = \begin{array}{c}
 \begin{tikzpicture}[yscale=-0.8, xscale=-0.8, point/.style={inner
 sep=0pt,circle,fill,minimum size=1mm}]
\begin{scope}[shift={(120:7)}]
\node at (3,0) [point] {};
\end{scope}
\begin{scope}[shift={(60:7)}]
\foreach \x in {0,30,60} { \node at (\x:-4.5) {$.$}; \node at
 (\x:-4.7) {$.$}; \node at (\x:-4.9) {$.$};}
\end{scope}
\foreach \y in {3,4,...,6}
{
\foreach \x in {3,...,\y}
{
\begin{scope}[shift={(120:\y)}]
\node (x\x\y) at (\x,0) [point] {};
\draw[->] (x\x\y) to ++(60:0.8);
\draw[->] ($(\x,0)+(120:1)$) to ($(\x,0)+(120:0.2)$);
\end{scope}
}
}
\begin{scope}[shift={(60:7)}]
\node[left] at (0,0) {$0$};
\node[below] at (-1,0) {$\varpi_1$};
\node[left] at (-120:1) {$\varpi_2$};
\end{scope}
\foreach \y in {4,5,6,7}
{
\foreach \x in {4,...,\y}
{
\begin{scope}[shift={(120:\y)}]
\node (x\x\y) at (\x,0) [point] {};
\draw[->] (\x,0) to (\x-0.8,0);
\end{scope}
}
}
 \end{tikzpicture}
\end{array}
\end{gather*}

Our algorithm consists of three steps. We begin with the labelled point $\big(0, 0, v^0\big)$ (our initial ``seed''). Each step in our algorithm enlarges our multiset in a new direction, starting with seeds generated in the previous step. The f\/irst step extends our set along a wall of the dominant chamber in the direction of $\varpi_1$ to produce a set~$X$. The second step extends our set along the walls of the $\ell$-alcoves to produce a multiset $Y$. Finally, the third step extends our multiset within the interior of each $\ell$-alcove to produce a multiset $Z$. In the third step, certain labelled points arising in the second step (seeds) gives rise to two spirals inside the interior of two adjacent alcoves which move like a
billiard. Finally we consider
\begin{gather*}
\widetilde{Z} := Z \setminus X,
\end{gather*}
which is the object of our conjecture.

\subsection{Step 1: the wall of the dominant chamber} Consider the full subgraph of $\Chi_+$ consisting of multiples of $\varpi_1$:
\begin{gather*}
\begin{tikzpicture}[scale=1,point/.style={inner sep=0pt,circle,fill,minimum size=1mm}]
\begin{scope}[color=black]
\foreach \x in {0,1, ..., 8} \node (x\x) at (\x,0) [point] {};
\foreach \x in {0,1, ..., 8} \draw[->] ($(x\x) + (0.2,0)$) to ($(x\x) + (0.8,0)$);
\end{scope}
\node[below] at (x0) {$0$};
\node[below] at (x1) {$\varpi_1$};
\end{tikzpicture}
\end{gather*}
Consider a labelled point $m = \big(\mu, m\big(v^k\big)\big)$ such that $\mu$ belongs to this full subgraph. A~\emph{small step} produces the labelled point $\big(\mu', (m+2)\big(v^k\big)\big)$ where $\mu \to \mu'$ is the unique edge with source~$\mu$. Consider the set~$X$ obtained by repeatedly taking small steps beginning with the labelled point $\big(0,0\big(v^0\big)\big)$. In other words, $X$ is the set
\begin{gather} \label{eq:X}
X = \big\{ \big( k \omega_1, 2k\big(v^0\big)\big) \, | \, k \in \ZM_{\ge 0} \big\}.
\end{gather}
We designate the labelled points of the form $(\ell k, 2k\ell(v^0))$ with $k > 0$ as \emph{seeds}.

\begin{Remark} Of course we could have def\/ined $X$ via \eqref{eq:X} directly. We prefer the inductive def\/inition, as it is closer in spirit to the more complicated def\/initions which will occur in the next two steps.
\end{Remark}

\subsection{Step 2: dynamics on the walls}\label{sec:wallbilliards}
From now on assume that $\ell \ge 3$. Consider the subgraph of $\Gamma$ with vertices $\l \in \Chi_+$ such that $\langle \l, \alpha^\vee \rangle \in \ell\ZM$ for some $\a \in\Phi_+$ and edges $(\l, \l')$ such that $\langle \l, \alpha^\vee \rangle = \langle \l', \alpha^\vee \rangle \in \ell \ZM$ for some $\a \in \Phi_+$. Let $\Gamma_\wall$ denote the graph obtained by removing the gray edges and vertices as in Fig.~\ref{fig:gammawall}.
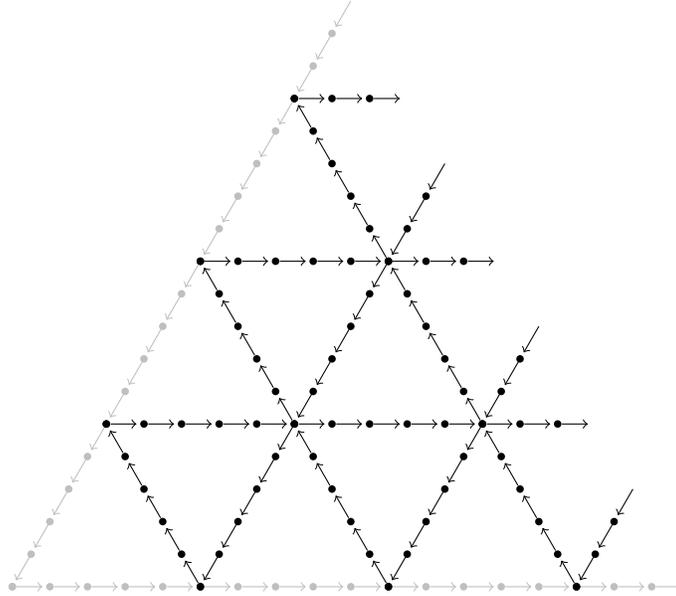
\begin{figure}\centering
\begin{tikzpicture}[scale=0.5,point/.style={inner sep=0pt,circle,fill,minimum size=1mm}]
\begin{scope}[color=gray!50]
\foreach \x in {0,1, ..., 17} \node (x\x) at (\x,0) [point] {};
\foreach \x in {0,1, ..., 17} \draw[->] (x\x) to ($(x\x) + (0.8,0)$);
\end{scope}
\begin{scope}[color=gray!50,shift={(0,0)},rotate=60]
\foreach \x in {0,1, ..., 17} \node (x\x) at (\x,0) [point] {};
\foreach \x in {0,1, ..., 17} \draw[->] ($(x\x) + (1,0)$) to ($(x\x) +
(0.2,0)$);
\end{scope}
\begin{scope}[shift={(60:5)}]
\foreach \x in {0,1, ..., 12} \node (x\x) at (\x,0) [point] {};
\foreach \x in {0,1, ..., 12} \draw[->] (x\x) to ($(x\x) + (0.8,0)$);
\end{scope}
\begin{scope}[shift={(60:10)}]
\foreach \x in {0,1, ..., 7} \node (x\x) at (\x,0) [point] {};
\foreach \x in {0,1, ..., 7} \draw[->] (x\x) to ($(x\x) + (0.8,0)$);
\end{scope}
\begin{scope}[shift={(60:15)}]
\foreach \x in {0,1, ..., 2} \node (x\x) at (\x,0) [point] {};
\foreach \x in {0,1, ..., 2} \draw[->] (x\x) to ($(x\x) + (0.8,0)$);
\end{scope}
\begin{scope}[shift={(5,0)},rotate=120]
\foreach \x in {0,1, ..., 5} \node (x\x) at (\x,0) [point] {};
\foreach \x in {0,1, ..., 4} \draw[->] (x\x) to ($(x\x) + (0.8,0)$);
\end{scope}
\begin{scope}[shift={(10,0)},rotate=120]
\foreach \x in {0,1, ..., 10} \node (x\x) at (\x,0) [point] {};
\foreach \x in {0,1, ..., 9} \draw[->] (x\x) to ($(x\x) + (0.8,0)$);
\end{scope}
\begin{scope}[shift={(15,0)},rotate=120]
\foreach \x in {0,1, ..., 15} \node (x\x) at (\x,0) [point] {};
\foreach \x in {0,1, ..., 14} \draw[->] (x\x) to ($(x\x) + (0.8,0)$);
\end{scope}
\begin{scope}[shift={(5,0)},rotate=60]
\foreach \x in {0,1, ..., 12} \node (x\x) at (\x,0) [point] {};
\foreach \x in {0,1, ..., 12} \draw[->] ($(x\x) + (1,0)$) to ($(x\x) +
(0.2,0)$);
\end{scope}
\begin{scope}[shift={(10,0)},rotate=60]
\foreach \x in {0,1, ..., 7} \node (x\x) at (\x,0) [point] {};
\foreach \x in {0,1, ..., 7} \draw[->] ($(x\x) + (1,0)$) to ($(x\x) +
(0.2,0)$);
\end{scope}
\begin{scope}[shift={(15,0)},rotate=60]
\foreach \x in {0,1, ..., 2} \node (x\x) at (\x,0) [point] {};
\foreach \x in {0,1, ..., 2} \draw[->] ($(x\x) + (1,0)$) to ($(x\x) +
(0.2,0)$);
\end{scope}
 \end{tikzpicture}
\caption{The directed graph $\Gamma_\wall$.} \label{fig:gammawall}
\end{figure}
Points in $\Chi_+$ or $\Gamma_\wall$ such that $\langle \lambda, \alpha^\vee \rangle \in \ell \ZM$ for all $\alpha \in \Phi_+$ are called \emph{corner points}. A~point $\mu$ in $\Chi_+$ of $\Gamma_\wall$ is an \emph{almost corner} if there exists a corner point~$c$ and an arrow $c \to \mu$ in~$\Chi_+$.

Fix a labelled point $\big(\mu,n\big(v^k\big)\big)$ with $\mu \in \Gamma_\wall$. We assume that~$\mu$ is such that there is a~unique edge with source~$\mu$. From $\big(\mu,n\big(v^k\big)\big)$ one may obtain new labelled points as follows:
\begin{enumerate}\itemsep=0pt
\item A \emph{rest} produces the labelled point $\big(\mu,(n+3)\big(v^{k+1}\big)\big)$.
\item A \emph{small step} produces (as above) the labelled point $\big(\mu', (n+2)\big(v^k\big)\big)$, where $\mu \to \mu'$ is the unique edge in $\Gamma_\wall$ with source $\mu$.
\item A \emph{giant leap} produces either one or two new labelled points, and is only possible if $\mu$ is not a corner or almost corner point. Def\/ine $d$ to be the direction of the unique arrow with source~$\mu$. First we proceed $j < \ell - 1$ steps in the graph~$\Gamma_\wall$ in direction $d$ until we reach a corner point. We then proceed for another~$\ell-1-j$ steps in all directions from the corner point which do not agree with the direction $d$. (There are either one or two such directions.) A giant leap consists of the resulting points, which are labelled by $(n+2\ell + 1)\big(v^{k+1}\big)$ (see Fig.~\ref{fig:giantleap}).
\end{enumerate}

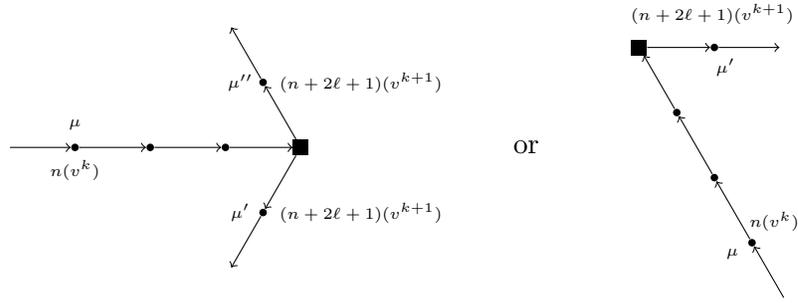
\begin{figure}\centering
 \begin{tikzpicture}[scale=1,point/.style={inner
 sep=0pt,circle,fill,minimum size=1mm},
corner/.style={inner sep=0pt,fill,minimum size=2mm}]
 \node (x1) at (1,0) {};
 \node (x2) at (2,0) [point] {};
 \node (x3) at (3,0) [point] {};
 \node (x4) at (4,0) [point] {};
 \node (x5) at (5,0) [corner] {};
\begin{scope}[shift={(5,0)},rotate=120]
 \node (x6) at (1,0) [point] {};
 \node (x7) at (2,0) {};
\end{scope}
\begin{scope}[shift={(5,0)},rotate=-120]
 \node (x6b) at (1,0) [point] {};
 \node (x7b) at (2,0) {};
\end{scope}
\draw[->] (x1) to (x2); \draw[->] (x2) to
(x3); \draw[->] (x3) to (x4); \draw[->] (x4) to (x5);
\draw[->] (x5) to (x6); \draw[->] (x6) to (x7);
\draw[->] (x5) to (x6b); \draw[->] (x6b) to (x7b);
\node at ($(x2) + (0,0.3)$) {\tiny $\mu$}; \node at ($(x2) - (0,0.3)$)
{\tiny $n(v^k)$};
\node at ($(x6) - (0.3,0)$) {\tiny $\mu''$}; \node at ($(x6) + (1.3,0)$)
{\tiny $(n+2\ell + 1)(v^{k+1})$};
\node at ($(x6b) - (0.3,0)$) {\tiny $\mu'$}; \node at ($(x6b) + (1.3,0)$)
{\tiny $(n+2\ell + 1)(v^{k+1})$};
\node at (8,0) {or};
\begin{scope}[shift={(12,-3)},rotate=120]
 \node (x1) at (1,0) {};
 \node (x2) at (2,0) [point] {};
 \node (x3) at (3,0) [point] {};
 \node (x4) at (4,0) [point] {};
 \node (x5) at (5,0) [corner] {};
\begin{scope}[shift={(5,0)},rotate=-120]
 \node (x6b) at (1,0) [point] {};
 \node (x7b) at (2,0) {};
\end{scope}
\draw[->] (x1) to (x2); \draw[->] (x2) to
(x3); \draw[->] (x3) to (x4); \draw[->] (x4) to (x5);
\draw[->] (x5) to (x6b); \draw[->] (x6b) to (x7b);
\node at ($(x2) + (0,0.3)$) {\tiny $\mu$}; \node at ($(x2) - (-0.1,0.4)$)
{\tiny $n(v^k)$};
\node at ($(x6b) - (0.3,0)$) {\tiny $\mu'$};
\node at ($(x6b) + (0.4,-0.2)$) {\tiny $(n+2\ell + 1)(v^{k+1})$};
\end{scope}
 \end{tikzpicture}
\caption{Performing a giant leap with $\ell = 5$. (Squares denote corner points.)} \label{fig:giantleap}
\end{figure}

Now suppose that we are given a labelled point $q = (\mu, n, v^k)$ as above (i.e., such that there is a unique edge in $\Gamma_\wall$ with source~$\mu$). We now describe a way of producing a multiset of labelled points (all with multiplicity~1) beginning with~$q$, some of which are designated as \emph{seeds}:
\begin{enumerate}\itemsep=0pt
\item \emph{If $\mu$ is a corner point.} We produce $\ell$ new points as follows: we take $\ell-1$ small steps to produce labelled points $m = m_1, \dots, m_{\ell - 1}$ and then rest to produce a new labelled point~$m_\ell$. The f\/inal labelled point $m_\ell$ is designated as a seed. (See Fig.~\ref{fig:baby}.)
\item \emph{If $\mu$ is an almost corner point.} We produce $\ell$ new points as follows: let $m = m_1$ and rest once to produce a new labelled point $m_2$, now take $\ell -2 $ small steps starting with~$m_2$ to produce $\ell - 2$ labelled points $m_3, \dots, m_{\ell -1}$ and f\/inally rest once more to produce a f\/inal labelled point $m_\ell$. The f\/inal labelled point $m_\ell$ is designated as a seed. (See Fig.~\ref{fig:baby}.)
\item \emph{If $\mu$ is neither a corner nor an almost corner point.} We take a giant leap to produce one or two new labelled points, each of which are designated seeds.
\end{enumerate}
We now iterate this process as follows. Starting with $q$ we produce a sequence of multisets
$Q_1, Q_2, Q_3, \dots$,
where $Q_1$ (resp.~$Q_i$ for $i > 1$) is obtained by applying the above procedure to~$q$ (resp.\ to each seed in $Q_{i-1}$). We say that the union (as multisets)
\begin{gather*}
Q = \bigcup_{i \ge 0} Q_i
\end{gather*}
is the result of \emph{applying dynamics on the walls to the seed $q$}.

\begin{figure}\centering
\begin{tikzpicture}[scale=1.2,point/.style={inner sep=0pt,circle,fill,minimum size=1mm}]
\begin{scope}[shift={(5,0)},rotate=120]
\foreach \x in {0,1, ..., 5} \node (x\x) at (\x,0) [point] {};
\foreach \x in {0,1, ..., 4} \draw[color=gray!50,->] (x\x) to ($(x\x) + (0.8,0)$);
\node[inner sep=0pt,left] (a1) at (0,0.15) {\tiny $\un{10(v^0)}$};
\node[left] (a2) at (x1) {\tiny $12(v^0)$};
\node[left] (a3) at (x2) {\tiny $14(v^0)$};
\node[left] (a4) at (x3) {\tiny $16(v^0)$};
\node[left] (a5) at (x4) {\tiny $18(v^0)$};
\node[inner sep=0pt,right] (a6) at (4,-0.15) {\tiny
 $\un{21(v^1)}$};
\draw[->] (a1) to (a2); \draw[->] (a2) to (a3); \draw[->] (a3) to
(a4); \draw[->] (a4) to (a5);
\end{scope}
\draw[->] (a5) to[out=70,in=110] (a6);
 \end{tikzpicture}
\qquad
\begin{tikzpicture}[scale=1.2,point/.style={inner sep=0pt,circle,fill,minimum size=1mm}]
\begin{scope}[shift={(5,0)},rotate=120]
\foreach \x in {0,1, ..., 5} \node (x\x) at (\x,0) [point] {};
\foreach \x in {0,1, ..., 4} \draw[color=gray!50,->] (x\x) to ($(x\x) + (0.8,0)$);
\node[inner sep=0pt,left] (a1) at (1,-.75) {\tiny $\un{54(v^4)}$};
\node[left] (a2) at (x1) {\tiny $57(v^5)$};
\node[left] (a3) at (x2) {\tiny $59(v^5)$};
\node[left] (a4) at (x3) {\tiny $61(v^5)$};
\node[left] (a5) at (x4) {\tiny $63(v^5)$};
\node[inner sep=0pt,right] (a6) at (4,-0.15) {\tiny
 $\un{66(v^6)}$};
\draw[->] (a2) to (a3); \draw[->] (a3) to
(a4); \draw[->] (a4) to (a5);
\end{scope}
\draw[->] (a1) to[out=-90,in=-60] (a2);
\draw[->] (a5) to[out=70,in=110] (a6);
 \end{tikzpicture}
\caption{An illustration of cases (1) and (2) when $\ell = 5$. Seeds are underlined.}\label{fig:baby}
\end{figure}
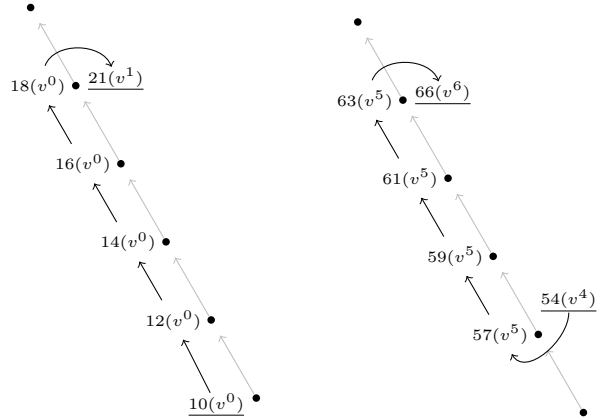

\begin{Example} Fig.~\ref{fig:l5ex} gives an example with $\ell = 5$. Our initial seed is the unique point with label $10\big(v^0\big)$. We illustrate a few
iterations of the above algorithm. Seeds are underlined. Note that one more iteration of the algorithm will produce the labelled point $\big( 3\varpi_1 + 7\varpi_2, 88\big(v^8\big)\big)$ with multiplicity~2.
\begin{figure}[t]\centering
\begin{tikzpicture}[scale=1.35,point/.style={inner sep=0pt,circle,fill,minimum size=1mm}]
\clip (2,-2) rectangle (11.5,12);
\begin{scope}[shift={(60:5)}]
\foreach \x in {0,1, ..., 12} \node (x\x) at (\x,0) [point] {};
\foreach \x in {0,1, ..., 12} \draw[color=gray!70, ->] (x\x) to ($(x\x) + (0.8,0)$);
\end{scope}
\node[above] at (x3) {\tiny $\un{32(v^2)}$};
\node[below] at (x3) {\tiny $\un{77(v^7)}$};
\begin{scope}[shift={(60:10)}]
\foreach \x in {0,1, ..., 7} \node (x\x) at (\x,0) [point] {};
\foreach \x in {0,1, ..., 7} \draw[color=gray!70,->] (x\x) to ($(x\x) + (0.8,0)$);
\end{scope}
\node[above] at (x1) {\tiny $\un{54(v^4)}$};
\node[below] at (x1) {\tiny $57(v^5)$};
\node[below] at (x2) {\tiny $59(v^5)$};
\node[below] at (x3) {\tiny $61(v^5)$};
\node[below] at (x4) {\tiny $63(v^5)$};
\node[above] at (x4) {\tiny $\un{66(v^6)}$};
\begin{scope}[shift={(60:15)}]
\foreach \x in {0,1, ..., 2} \node (x\x) at (\x,0) [point] {};
\foreach \x in {0,1, ..., 2} \draw[color=gray!70,->] (x\x) to ($(x\x) + (0.8,0)$);
\end{scope}
\begin{scope}[shift={(5,0)},rotate=120]
\foreach \x in {0,1, ..., 5} \node (x\x) at (\x,0) [point] {};
\foreach \x in {0,1, ..., 4} \draw[color=gray!70,->] (x\x) to ($(x\x) + (0.8,0)$);
\end{scope}
\node[inner sep=0pt,left] at (x0) {\tiny $\un{10(v^0)}$};
\node[below left] at (x1) {\tiny $12(v^0)$};
\node[below left] at (x2) {\tiny $14(v^0)$};
\node[below left] at (x3) {\tiny $16(v^0)$};
\node[below left] at (x4) {\tiny $18(v^0)$};
\node[left] at (x4) {\tiny $\un{21(v^1)}$};
\node[right] at (x1) {\tiny $\un{54(v^4)}$};
\node[above right] at (x1) {\tiny $57(v^5)$};
\node[above right] at (x2) {\tiny $59(v^5)$};
\node[above right] at (x3) {\tiny $61(v^5)$};
\node[right] at (x4) {\tiny $63(v^5)$};
\node[above right] at (x4) {\tiny $\un{66(v^6)}$};
\begin{scope}[shift={(10,0)},rotate=120]
\foreach \x in {0,1, ..., 10} \node (x\x) at (\x,0) [point] {};
\foreach \x in {0,1, ..., 9} \draw[color=gray!70,->] (x\x) to ($(x\x) + (0.8,0)$);
\end{scope}
\node[left] at (x7) {\tiny $\un{43(v^3)}$};
\begin{scope}[shift={(15,0)},rotate=120]
\foreach \x in {0,1, ..., 15} \node (x\x) at (\x,0) [point] {};
\foreach \x in {0,1, ..., 14} \draw[color=gray!70,->] (x\x) to ($(x\x) + (0.8,0)$);
\end{scope}
\node[right] at (x13) {\tiny $\un{77(v^7)}$};
\begin{scope}[shift={(5,0)},rotate=60]
\foreach \x in {0,1, ..., 12} \node (x\x) at (\x,0) [point] {};
\foreach \x in {0,1, ..., 12} \draw[color=gray!70,->] ($(x\x) + (1,0)$) to ($(x\x) +
(0.2,0)$);
\end{scope}
\node[right] at (x3) {\tiny $\un{43(v^3)}$};
\node[right] at (x7) {\tiny $\un{77(v^7)}$};
\begin{scope}[shift={(10,0)},rotate=60]
\foreach \x in {0,1, ..., 7} \node (x\x) at (\x,0) [point] {};
\foreach \x in {0,1, ..., 7} \draw[color=gray!70,->] ($(x\x) + (1,0)$) to ($(x\x) +
(0.2,0)$);
\end{scope}
\begin{scope}[shift={(15,0)},rotate=60]
\foreach \x in {0,1, ..., 2} \node (x\x) at (\x,0) [point] {};
\foreach \x in {0,1, ..., 2} \draw[color=gray!70,->] ($(x\x) + (1,0)$) to ($(x\x) +
(0.2,0)$);
\end{scope}
 \end{tikzpicture}\vspace{-23mm}

\caption{Dynamics on the walls with $\ell = 5$.}\label{fig:l5ex}
\end{figure}
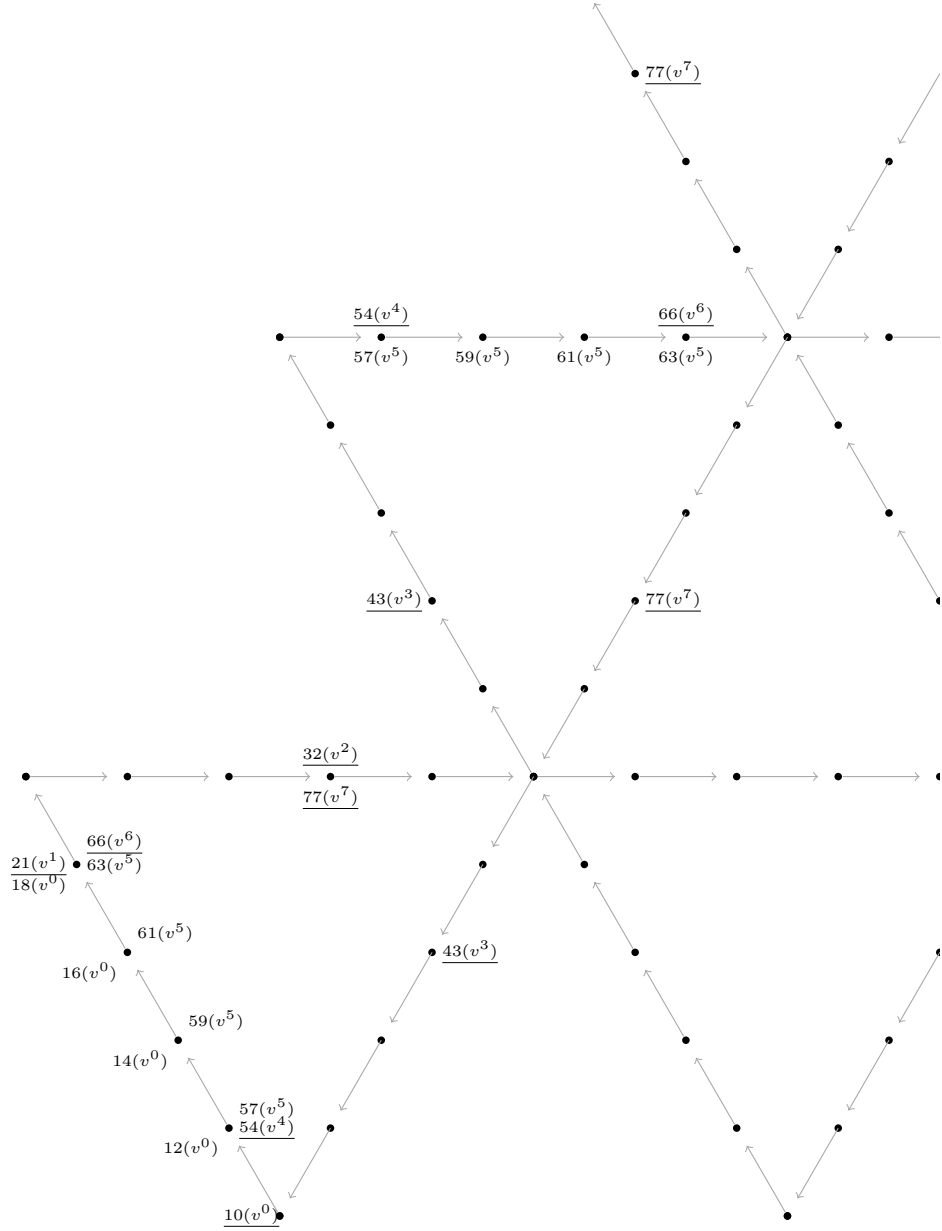
\end{Example}

Now consider our set $X$ from the previous step. We apply dynamicson walls to each seed in~$X$ (i.e., each labelled point of the form $\big(k\ell\varpi_1, 2k\ell\big(v^0\big)\big)$ with $k > 0$) to generate a multiset~$Y_k$. We now def\/ine a~new multiset{\samepage
\begin{gather*}
Y := X \cup \bigcup_{k > 0} Y_k.
\end{gather*}
We also remember (for the purposes of the next step) which elements of $Y$ were designated seeds.}

\begin{Remark} \label{rem:growth} Consider dynamics on the wall restricted to the full subgraph displayed in Fig.~\ref{fig:l5growth} with seed $q = \big(\mu, n\big(v^k\big)\big)$. Then $4$ iterations of the above algorithm yields the following points and labels (see Fig.~\ref{fig:l5growth}):
 \begin{gather*}
 u, u' \quad \text{with label} \quad (n+11)\big(v^{k+1}\big), \\
 u_i, u_i' \quad \text{with label} \quad (n+14 + 2(i-1))
 \big(v^{k+2}\big) \quad \text{for }{i = 1,2,3,4},\\
 v, v' \quad \text{with label} \quad (n+23)\big(v^{k+3}\big), \\
 w, w' \quad \text{with label} \quad (n+34)\big(v^{k+4}\big),
 \end{gather*}
and \looseness=-1 then $\mu$ with label $(n+45)\big(v^{k+5}\big)$ and multiplicity~2. (The point is that there are two directed paths leading from $\mu$ back to itself, which causes the multiplicity to double.) Repeating this algorithm $4i$ times leads to $q_i := \big(\mu, n+45i\big(v^{k+5i}\big)\big)$ with multiplicity~$2^i$. From this observation one deduces easily that for $\ell = 5$ the set $Y$ contains labelled points whose multiplicity grows exponentially in~$n$. Similar considerations show that the same statement about~$Y$ is true for any~$\ell$.
\end{Remark}

\begin{figure}[t]\centering
\begin{tikzpicture}[scale=.7,point/.style={inner
 sep=0pt,circle,fill,minimum size=1mm}]
\foreach \x in {0,1, ..., 5} \node (x\x) at (\x,0) [point] {};
\foreach \x in {0,1, ..., 4} \draw[color=gray!70, ->] ($(x\x) +
(0.2,0)$) to ($(x\x) + (0.8,0)$);
\node[below] at (x2) {$q$};
\begin{scope}[rotate=60]
\foreach \x in {0,1, ..., 5} \node (x\x) at (\x,0) [point] {};
\foreach \x in {0,1, ..., 4} \draw[color=gray!70, <-] ($(x\x) +
(0.2,0)$) to ($(x\x) + (0.8,0)$);
\node [above left] at (x2) {$w$};
\end{scope}
\begin{scope}[shift={(0:5)},rotate=120]
\foreach \x in {0,1, ..., 5} \node (x\x) at (\x,0) [point] {};
\foreach \x in {0,1, ..., 4} \draw[color=gray!70,->] ($(x\x) +
(0.2,0)$) to ($(x\x) + (0.8,0)$);
\node [above right] at (x1) {$u$};
\node [left] at (x1) {$u_1$}; \node [left] at (x2) {$u_2$};
\node [left] at (x3) {$u_3$}; \node [left] at (x4) {$u_4$};
\node [above right] at (x4) {$v$};
\end{scope}
\begin{scope}[rotate=300]
\foreach \x in {0,1, ..., 5} \node (x\x) at (\x,0) [point] {};
\foreach \x in {0,1, ..., 4} \draw[color=gray!70, <-] ($(x\x) +
(0.2,0)$) to ($(x\x) + (0.8,0)$);
\node [below left] at (x2) {$w'$};
\end{scope}
\begin{scope}[shift={(0:5)},rotate=240]
\foreach \x in {0,1, ..., 5} \node (x\x) at (\x,0) [point] {};
\foreach \x in {0,1, ..., 4} \draw[color=gray!70,->] ($(x\x) +
(0.2,0)$) to ($(x\x) + (0.8,0)$);
\node [below right] at (x1) {$u'$};
\node [left] at (x1) {$u_1'$}; \node [left] at (x2) {$u_2'$};
\node [left] at (x3) {$u_3'$}; \node [left] at (x4) {$u_4'$};
\node [below right] at (x4) {$v'$};
\end{scope}
 \end{tikzpicture}
\caption{Subgraph demonstrating exponential growth of multiplicities for $\ell = 5$.}\label{fig:l5growth}
\end{figure}
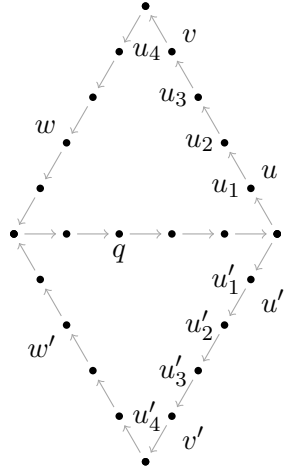

\subsection{Step 3: billiards in an alcove} \label{sec:alcovebilliards}

We now describe an algorithm which produces for each seed in $Y$ a~multiset of labelled points. We need a little more notation. We consider the dominant weights $\Chi_+$ as embedded in the vector space
\begin{gather*}
\Chi_\RM := \Chi \otimes_{\ZM} \RM,\end{gather*}
which we regard as a Euclidean space for some $\Wf$-invariant bilinear form. An \emph{$\ell$-alcove} is a~connected component of the complement
\begin{gather*}\Chi_{\RM} \setminus \bigcup_{\a \in \Phi_+^\vee} \big\{ \l \, | \, \langle \alpha^\vee, \l \rangle \in \ell \ZM \big\}.\end{gather*}
An \emph{$\ell$-regular point} is a point which belongs to some $\ell$-alcove. An $\ell$-alcove is \emph{dominant} if it is contained in the real cone generated by $\Chi_+$.

Now consider a seed $q = (\mu, m) \in Y$. We want to associate a multiset of labelled points $\Lambda_q$ to $q$. If $\mu$ is a corner or almost corner, we set $\Lambda_q := \varnothing$.

From now on we assume that $q$ is neither a corner or almost corner. In this case $\mu$ belongs to the closure of two dominant $\ell$-alcoves, $A'$ and $A''$. We perform an identical procedure for both alcoves, so f\/ix one such alcove and call it $A$. Let $\Delta$ denote the full subgraph of $\Gamma$ consisting of points belonging to $\overline{A}$. Then $\Delta$ has one of the following forms:
\begin{gather*}
 \begin{tikzpicture}[scale=0.7, point/.style={inner
 sep=0pt,circle,fill,minimum size=1mm}]
\draw [decorate,decoration={brace,amplitude=10pt},xshift=-10pt,yshift=-5pt]
(0,0) to (120:7);
\node at (-3,2.4) {\tiny $\ell + 1$};
\node at (-3,2) {\tiny points};
\begin{scope}[xscale=-1]
\node at (120:7) [point] {};
\foreach \y in {0,1,...,6}
{
\foreach \x in {0,...,\y}
{
\begin{scope}[shift={(120:\y)}]
\node (x\x\y) at (\x,0) [point] {};
\draw[->] (x\x\y) to ++(60:0.8);
\draw[->] ($(\x,0)+(120:1)$) to ($(\x,0)+(120:0.2)$);
\end{scope}
}
}
\foreach \y in {1,2,...,7}
{
\foreach \x in {1,...,\y}
{
\begin{scope}[shift={(120:\y)}]
\node (x\x\y) at (\x,0) [point] {};
\draw[->] (\x,0) to (\x-0.8,0);
\end{scope}
}
}
\end{scope}
\begin{scope}[scale=-1,shift={(7,-6)}]
\node at (120:7) [point] {};
\foreach \y in {0,1,...,6}
{
\foreach \x in {0,...,\y}
{
\begin{scope}[shift={(120:\y)}]
\node (x\x\y) at (\x,0) [point] {};
\draw[->] (x\x\y) to ++(60:0.8);
\draw[->] ($(\x,0)+(120:1)$) to ($(\x,0)+(120:0.2)$);
\end{scope}
}
}
\foreach \y in {1,2,...,7}
{
\foreach \x in {1,...,\y}
{
\begin{scope}[shift={(120:\y)}]
\node (x\x\y) at (\x,0) [point] {};
\draw[->] (\x,0) to (\x-0.8,0);
\end{scope}
}
}
\end{scope}
 \end{tikzpicture}
\end{gather*}
Points of $\Delta$ belong to $\overline{A}$ but not to $A$ are called \emph{wall points}. Note that our starting point $\mu$ is a~wall point.

\emph{Variables and their initialisations:} Let $(\mu_\start, m_\start) := (\mu,m)$ and let $d_\start$ denote the unique edge of~$\Delta$ with source $\mu$ and target an interior point of $\Delta$. To begin with, set $\Lambda := \{ (\mu_\start, m_\start) \}$. In the algorithm we need the following variables, which are initialised as follows:
\begin{alignat*}{3}
& \curpoint := \mu_{\textrm{start}}, \qquad && \curlabel := m_{\textrm{start}},& \\
& \curdir := d_{\textrm{start}},\qquad && \cursign := 1. &
\end{alignat*}

\emph{The loop:} The multiset of labelled points is obtained by repeating the following ad inf\/initum. Let $\mu_\new$ denote the point obtained by moving from $\curpoint$ one step in $\curdir$. Write $\curlabel$ as $n\big(v^k\big)$. Two possibilities may occur:
\begin{enumerate}\itemsep=0pt
\item $\mu_{\new}$ is $\ell$-regular:
\begin{gather*}
\curlabel := (n+2)\big(v^{k}\big),\qquad \curpoint := \mu_\new,\\
\text{add $(\curpoint, \curlabel)$ to $\Lambda$}.
\end{gather*}
(The variables $\curdir$ and $\cursign$ remain unchanged.)
\item $\curpoint$ is $\ell$-regular and $\mu_{\new}$ is a
 wall point:
\begin{alignat*}{3}
& \curlabel := (n+3)\big(v^{k + \cursign}\big),\qquad\!\! && \curdir := d_\new,& \\
& \cursign := -\cursign, \qquad\!\! && \text{add $(\curpoint, \curlabel)$ to $\Lambda$}.&
\end{alignat*}
Here $d_\new$ denotes the vector obtained by ref\/lecting $\curdir$ in the wall on which~$\mu_\new$ lies. (The variable $\curpoint$ remains unchanged.)
\end{enumerate}

\begin{Example} We illustrate the f\/irst few steps of the algorithm in some examples. In each case the starting point $\mu_\start$ and its label are underlined. The successive values of $\curlabel$ and $\curpoint$ are obtained by following the arrows.

\begin{enumerate}\itemsep=0pt
\item A reasonably generic example with $\ell = 11$:
\begin{gather*}
\resizebox{12cm}{!}{
\begin{tikzpicture}
[scale=1.3, point/.style={inner sep=0pt,circle,fill,minimum size=1mm,color=gray}]
%
\node (x00) at (0,0) [point] {}; \node[fill] (x10) at (1,0) [point] {}; \node (x20) at (2,0) [point] {}; \node[fill] (x30) at (3,0) [point] {};
\node (x40) at (4,0) [point] {}; \node[fill] (x50) at (5,0) [point] {}; \node (x60) at (6,0) [point] {}; \node[fill] (x70) at (7,0) [point] {};
\node (x80) at (8,0) [point] {}; \node[fill] (x90) at (9,0) [point] {}; \node (x100) at (10,0) [point] {}; \node[fill] (x110) at (11,0) [point] {};
\begin{scope}[shift={(-60:1)}]
\node (x01) at (0,0) [point] {}; \node[fill] (x11) at (1,0) [point] {}; \node (x21) at (2,0) [point] {}; \node[fill] (x31) at (3,0) [point] {};
\node (x41) at (4,0) [point] {}; \node[fill] (x51) at (5,0) [point] {}; \node (x61) at (6,0) [point] {}; \node[fill] (x71) at (7,0) [point] {};
\node (x81) at (8,0) [point] {}; \node[fill] (x91) at (9,0) [point] {}; \node (x101) at (10,0) [point] {};
 \end{scope}
\begin{scope}[shift={(-60:2)}]
\node (x02) at (0,0) [point] {}; \node[fill] (x12) at (1,0) [point] {}; \node (x22) at (2,0) [point] {}; \node[fill] (x32) at (3,0) [point] {};
\node (x42) at (4,0) [point] {}; \node[fill] (x52) at (5,0) [point] {}; \node (x62) at (6,0) [point] {}; \node[fill] (x72) at (7,0) [point] {};
\node (x82) at (8,0) [point] {}; \node[fill] (x92) at (9,0) [point] {};
 \end{scope}
\begin{scope}[shift={(-60:3)}]
\node[fill] (x13) at (1,0) [point] {}; \node (x23) at (2,0) [point] {}; \node[fill] (x33) at (3,0) [point] {};
\node (x43) at (4,0) [point] {}; \node[fill] (x53) at (5,0) [point] {}; \node (x63) at (6,0) [point] {}; \node[fill] (x73) at (7,0) [point] {};
\node (x83) at (8,0) [point] {};
 \end{scope}
\begin{scope}[shift={(-60:4)}]
\node (x04) at (0,0) [point] {}; \node[fill] (x14) at (1,0) [point] {}; \node (x24) at (2,0) [point] {}; \node[fill] (x34) at (3,0) [point] {};
\node (x44) at (4,0) [point] {}; \node[fill] (x54) at (5,0) [point] {}; \node (x64) at (6,0) [point] {}; \node[fill] (x74) at (7,0) [point] {};
 \end{scope}
\begin{scope}[shift={(-60:5)}]
\node (x05) at (0,0) [point] {}; \node[fill] (x15) at (1,0) [point] {}; \node (x25) at (2,0) [point] {}; \node[fill] (x35) at (3,0) [point] {};
\node (x45) at (4,0) [point] {}; \node[fill] (x55) at (5,0) [point] {}; \node (x65) at (6,0) [point] {};
 \end{scope}
\begin{scope}[shift={(-60:6)}]
\node (x06) at (0,0) [point] {}; \node[fill] (x16) at (1,0) [point] {}; \node (x26) at (2,0) [point] {}; \node[fill] (x36) at (3,0) [point] {};
\node (x46) at (4,0) [point] {}; \node[fill] (x56) at (5,0) [point] {};
 \end{scope}
\begin{scope}[shift={(-60:7)}]
\node (x07) at (0,0) [point] {}; \node[fill] (x17) at (1,0) [point] {}; \node (x27) at (2,0) [point] {}; \node[fill] (x37) at (3,0) [point] {};
\node (x47) at (4,0) [point] {};
 \end{scope}
\begin{scope}[shift={(-60:8)}]
\node (x08) at (0,0) [point] {}; \node[fill] (x18) at (1,0) [point] {}; \node (x28) at (2,0) [point] {}; \node[fill] (x38) at (3,0) [point] {};
 \end{scope}
\begin{scope}[shift={(-60:9)}]
\node (x09) at (0,0) [point] {}; \node[fill] (x19) at (1,0) [point] {}; \node (x29) at (2,0) [point] {};
 \end{scope}
\begin{scope}[shift={(-60:10)}]
\node (x010) at (0,0) [point] {}; \node[fill] (x110) at (1,0) [point] {};
 \end{scope}
\begin{scope}[shift={(-60:11)}]
\node (x011) at (0,0) [point] {};
 \end{scope}
\node (n0) at ($(x13)+(0.2,0.1)$) {\small $(*)$};
\node (n0) at ($(x13)+(0.2,0.2)$) {};
\node (n1) at ($(x13)+(0,-0.2)$) {\tiny $16(v)$};
\node (n2) at ($(x23)+(0,-0.2)$) {\tiny $18(v)$};
\node (n3) at ($(x33)+(0,-0.2)$) {\tiny $20(v)$};
\node (n4) at ($(x43)+(0,-0.2)$) {\tiny $22(v)$};
\node (n5) at ($(x53)+(0,-0.2)$) {\tiny $24(v)$};
\node (n6) at ($(x63)+(0,-0.2)$) {\tiny $26(v)$};
\node (n7) at ($(x73)+(0,-0.2)$) {\tiny $28(v)$};
\node (n8) at ($(x73)+(0.3,0.2)$) {\tiny $31(v^2)$};
\node (n9) at ($(x72)+(0.3,0.2)$) {\tiny $33(v^2)$};
\node (n10) at ($(x71)+(0.3,0.2)$) {\tiny $35(v^2)$};
\node (n11) at ($(x71)+(-0.3,0.2)$) {\tiny $38(v)$};
\node (n12) at ($(x62)+(-0.2,0.2)$) {\tiny $40(v)$};
\node (n13) at ($(x53)+(-0.2,0.2)$) {\tiny $42(v)$};
\node (n14) at ($(x44)+(-0.2,0.2)$) {\tiny $44(v)$};
\node (n15) at ($(x35)+(-0.2,0.2)$) {\tiny $46(v)$};
\node (n16) at ($(x26)+(-0.2,0.2)$) {\tiny $48(v)$};
\node (n17) at ($(x17)+(-0.3,0.2)$) {\tiny $50(v)$};
\node (n18) at ($(x17)-(0,0.2)$) {\tiny $53(v^2)$};
\node (n19) at ($(x27)-(0,0.2)$) {\tiny $55(v^2)$};
\node (n20) at ($(x37)-(0,0.2)$) {\tiny $57(v^2)$};
\node (n21) at ($(x37)+(0.3,0.2)$) {\tiny $60(v)$};
\node (n22) at ($(x36)+(0.3,0.2)$) {\tiny $62(v)$};
\node (n23) at ($(x35)+(0.3,0.2)$) {\tiny $64(v)$};
\node (n24) at ($(x34)+(0.3,0.2)$) {\tiny $66(v)$};
\node (n25) at ($(x33)+(0.3,0.2)$) {\tiny $68(v)$};
\node (n26) at ($(x32)+(0.3,0.2)$) {\tiny $70(v)$};
\node (n27) at ($(x31)+(0.3,0.2)$) {\tiny $72(v)$};
\node (n28) at ($(x31)+(-0.3,0.2)$) {\tiny $75(v^2)$};
\node (n29) at ($(x22) +(-0.3,0.2)$) {\tiny $77(v^2)$};
\node (n30) at ($(x13) +(-0.3,0.2)$) {\tiny $\dots$};
\node[inner sep=1pt] (seed) at (-60:3.2) {\tiny $\un{14(v)}$};
\draw[->] (seed) to (n1);
\draw[->] (n1) -- (n2);
\draw[->] (n2) -- (n3);
\draw[->] (n3) -- (n4);
\draw[->] (n4) -- (n5);
\draw[->] (n5) -- (n6);
\draw[->] (n6) -- (n7);
\draw[->] (n7) to [out=0,in=-60] (n8);
\draw[->] (n8) -- (n9);
\draw[->] (n9) -- (n10);
\draw[->] (n10) to [out=120,in=60] (n11);
\draw[->] (n11) -- (n12);
\draw[->] (n12) -- (n13);
\draw[->] (n13) -- (n14);
\draw[->] (n14) -- (n15);
\draw[->] (n15) -- (n16);
\draw[->] (n16) -- (n17);
\draw[->] (n17) to [out=-120,in=180] (n18);
\draw[->] (n18) -- (n19);
\draw[->] (n19) -- (n20);
\draw[->] (n20) to [out=0,in=-60] (n21);
\draw[->] (n21) -- (n22);
\draw[->] (n22) -- (n23);
\draw[->] (n23) -- (n24);
\draw[->] (n24) -- (n25);
\draw[->] (n25) -- (n26);
\draw[->] (n26) -- (n27);
\draw[->] (n27) to [out=120,in=60] (n28);
\draw[->] (n28) -- (n29);
\draw[->] (n29) -- (n30);
 \end{tikzpicture}
}
\end{gather*}
If we focus on the point marked with an asterix (*) we obtain the labels:
\begin{gather*}
\big\{ 16(v), 79\big(v^2\big), 82(v), 145\big(v^2\big), 148(v), 211\big(v^2\big), 214(v), 277\big(v^2\big), 280(v), \dots \big\}.
\end{gather*}
\item An interesting example with $\ell = 5$ (we display both alcoves):
\begin{gather*}
\resizebox{14cm}{!}{
\begin{tikzpicture}
[xscale=3.5,yscale=3.5, point/.style={inner sep=0pt,circle,fill,minimum size=1mm,color=gray}]
\node (x00) at (0,0) [point] {}; \node[fill] (x10) at (1,0) [point] {}; \node (x20) at (2,0) [point] {}; \node[fill] (x30) at (3,0) [point] {};
\node (x40) at (4,0) [point] {}; \node[fill] (x50) at (5,0) [point] {};
\begin{scope}[shift={(60:1)}]
\node (x01) at (0,0) [point] {}; \node[fill] (x11) at (1,0) [point] {}; \node (x21) at (2,0) [point] {}; \node[fill] (x31) at (3,0) [point] {};
\node (x41) at (4,0) [point] {};
\begin{scope}[shift={(1,0)}]
\node (a3) at (150:0.2) {$*$};
\node (a4) at (-150:0.2) {$*$};
\node (a5) at (-90:0.2) {$*$};
\end{scope}
\begin{scope}[shift={(2,0)}]
\node (a6) at (-90:0.2) {$*$};
\end{scope}
\begin{scope}[shift={(3,0)}]
\node (a7) at (-90:0.2) {$*$};
\node (a8) at (-30:0.2) {$*$};
\node (a9) at (30:0.2) {$*$};
\end{scope}
 \end{scope}
\begin{scope}[shift={(60:2)}]
\node (x02) at (0,0) [point] {}; \node[fill] (x12) at (1,0) [point]
{}; \node (x22) at (2,0) [point] {}; \node[fill] (x32) at (3,0)
[point] {};
\begin{scope}[shift={(1,0)}]
\node (a2) at (150:0.2) {$*$};
\node (a2b) at (150:0.1) {\rotatebox{60}{$\dots$}};
\end{scope}
\begin{scope}[shift={(2,0)}]
\node (a10) at (30:0.2) {$*$};
\end{scope}
 \end{scope}
\begin{scope}[shift={(60:3)}]
\node (x03) at (0,0) [point] {}; \node[fill] (x13) at (1,0) [point]
{}; \node (x23) at (2,0) [point] {};
\begin{scope}[shift={(1,0)}]
\node (a1) at (150:0.2) {$*$};
\node (a1b) at (150:0.1) {$*$};
\node (a12) at (75:0.15) {$*$};
\node (a11) at (20:0.2) {$*$};
\end{scope}
 \end{scope}
\begin{scope}[shift={(60:4)}]
\node (x04) at (0,0) [point] {}; \node (x14) at (1,0) {};
 \end{scope}
\begin{scope}[shift={(60:5)}]
\node (x05) at (0,0) [point] {};
 \end{scope}
\draw[->] (x14) to[out=-150,in=60] (a1);
\draw[->] (a1) to (a2); \draw[->] (a2) to (a3);
\draw[->] (a3) to (a4); \draw[->] (a4) to (a5);
\draw[->] (a5) to (a6); \draw[->] (a6) to (a7);
\draw[->] (a7) to (a8); \draw[->] (a8) to (a9);
\draw[->] (a9) to (a10); \draw[->] (a10) to (a11);
\draw[->] (a11) to (a12); \draw[->] (a12) to (a1b);
\draw[->] (a1b) to (a2b);
\draw[->,color=gray!60] (x50) -- (x41);
\draw[->,color=gray!60] (x41) -- (x32);
\draw[->,color=gray!60] (x32) -- (x23);
\draw[->,color=gray!60] (x23) -- (x14);
\draw[->,color=gray!60] (x14) -- (x05);
\draw[->,color=gray!60] (x50) -- (x40);
\draw[->,color=gray!60] (x40) -- (x30);
\draw[->,color=gray!60] (x30) -- (x20);
\draw[->,color=gray!60] (x20) -- (x10);
\draw[->,color=gray!60] (x10) -- (x00);
\draw[->,color=gray!60] (x00) -- (x05);
%
\node[fill=white] (n0) at (x14) { $\un{21(v)}$};
\begin{scope}[shift={(60:5)},xscale=1,yscale=-1]
%
\node (x00) at (0,0) [point] {}; \node[fill] (x10) at (1,0) [point] {}; \node (x20) at (2,0) [point] {}; \node[fill] (x30) at (3,0) [point] {};
\node (x40) at (4,0) [point] {}; \node[fill] (x50) at (5,0) [point] {};
\begin{scope}[shift={(60:1)}]
\node (x01) at (0,0) {}; \node[fill] (x11) at (1,0) [point] {}; \node (x21) at (2,0) [point] {}; \node[fill] (x31) at (3,0) [point] {};
\node (x41) at (4,0) [point] {};
\begin{scope}[shift={(1,0)}]
\node (a3) at (150:0.2) {$*$};
\node (a4) at (-160:0.15) {$*$};
\node (a5) at (-90:0.2) {$*$};
\node (a5b) at (-90:0.1) {$*$};
\end{scope}
\begin{scope}[shift={(2,0)}]
\node (a6) at (-90:0.2) {$*$};
\node (a6b) at (-90:0.1) {$\dots$};
\end{scope}
\begin{scope}[shift={(3,0)}]
\node (a7) at (-90:0.2) {$*$};
\node (a8) at (-30:0.2) {$*$};
\node (a9) at (30:0.2) {$*$};
\end{scope}
 \end{scope}
\begin{scope}[shift={(60:2)}]
\node (x02) at (0,0) [point] {}; \node[fill] (x12) at (1,0) [point]
{}; \node (x22) at (2,0) [point] {}; \node[fill] (x32) at (3,0)
[point] {};
\begin{scope}[shift={(1,0)}]
\node (a2) at (150:0.2) {$*$};
\end{scope}
\begin{scope}[shift={(2,0)}]
\node (a10) at (30:0.2) {$*$};
\end{scope}
 \end{scope}
\begin{scope}[shift={(60:3)}]
\node (x03) at (0,0) [point] {}; \node[fill] (x13) at (1,0) [point]
{}; \node (x23) at (2,0) [point] {};
\begin{scope}[shift={(1,0)}]
\node (a1) at (150:0.2) {$*$};
\node (a12) at (90:0.2) {$*$};
\node (a11) at (30:0.2) {$*$};
\end{scope}
 \end{scope}
\begin{scope}[shift={(60:4)}]
\node (x04) at (0,0) [point] {}; \node (x14) at (1,0) [point] {};
 \end{scope}
\begin{scope}[shift={(60:5)}]
\node (x05) at (0,0) [point] {};
 \end{scope}
\draw[->] (n0) to[out=-30,in=-180] (25:1.58);
\draw[->] (a1) to (a2); \draw[->] (a2) to (a3);
\draw[->] (a3) to (a4); \draw[->] (a4) to (a5b);
\draw[->] (a5) to (a6); \draw[->] (a6) to (a7);
\draw[->] (a7) to (a8); \draw[->] (a8) to (a9);
\draw[->] (a9) to (a10); \draw[->] (a10) to (a11);
\draw[->] (a11) to (a12); \draw[->] (a12) to (a1);
\draw[->] (a5b) to (a6b);
\draw[->,color=gray!60] (x50) -- (x40);
\draw[->,color=gray!60] (x40) -- (x30);
\draw[->,color=gray!60] (x30) -- (x20);
\draw[->,color=gray!60] (x20) -- (x10);
\draw[->,color=gray!60] (x10) -- (x00);
\draw[->,color=gray!60] (x50) -- (x05);
\end{scope}
 \end{tikzpicture}
}
\end{gather*}
In either alcove the sequence of labels obtained by following the arrows beginning at $21(v)$ is as follows:
\begin{gather*}
23(v), 25(v), 27(v), 30\ig(v^2\big), 33(v), 35(v), 37(v), 40\big(v^2\big),\\
\qquad 43(v), 45(v), 47(v), 50\big(v^2\big), 53(v), \dots\,.
\end{gather*}
\item
Another example with $\ell = 5$ (again we display both alcoves):
\begin{gather*}
\resizebox{14cm}{!}{
\begin{tikzpicture}
[xscale=3.5,yscale=3.5, point/.style={inner sep=0pt,circle,fill,minimum size=1mm,color=gray}]
\node (x00) at (0,0) [point] {}; \node[fill] (x10) at (1,0) [point] {}; \node (x20) at (2,0) [point] {}; \node[fill] (x30) at (3,0) [point] {};
\node (x40) at (4,0) [point] {}; \node[fill] (x50) at (5,0) [point] {};
\begin{scope}[shift={(60:1)}]
\node (x01) at (0,0) [point] {}; \node[fill] (x11) at (1,0) [point] {}; \node (x21) at (2,0) [point] {}; \node[fill] (x31) at (3,0) [point] {};
\node (x41) at (4,0) [point] {};
\begin{scope}[shift={(2,0)}]
\node (a2) at (-60:0.2) {$*$};
\node (a2b) at (-90:0.075) {$\rotatebox{0}{\dots}$};
\node (a3) at (-120:0.2) {$*$};
\end{scope}
 \end{scope}
\begin{scope}[shift={(60:2)}]
\node (x02) at (0,0) [point] {}; \node[fill] (x12) at (1,0) [point]
{}; \node (x22) at (2,0) [point] {}; \node[fill] (x32) at (3,0)
[point] {};
\begin{scope}[shift={(1,0)}]
\node (a4) at (180:0.2) {$*$};
\node (a5) at (120:0.2) {$*$};
\end{scope}
\begin{scope}[shift={(2,0)}]
\node (a1b) at (0:0.2) {$*$};
\node (a1a) at (0:0.3) {$*$};
\node (a6) at (60:0.2) {$*$};
\end{scope}
 \end{scope}
\begin{scope}[shift={(60:3)}]
\node (x03) at (0,0) [point] {}; \node[fill] (x13) at (1,0) [point]
{}; \node (x23) at (2,0) {};
 \end{scope}
\begin{scope}[shift={(60:4)}]
\node (x04) at (0,0) [point] {}; \node (x14) at (1,0) [point] {};
 \end{scope}
\begin{scope}[shift={(60:5)}]
\node (x05) at (0,0) [point] {};
 \end{scope}
\draw[->] (a1a) to (a2); \draw[->] (a2) to (a3);
\draw[->] (a3) to (a4); \draw[->] (a4) to (a5);
\draw[->] (a5) to (a6); \draw[->] (a6) to (a1b);
\draw[->] (a1b) to ($(a2b)+(0.09,0.05)$);
\draw[->,color=gray!60] (x50) -- (x41);
\draw[->,color=gray!60] (x41) -- (x32);
\draw[->,color=gray!60] (x32) -- (x23);
\draw[->,color=gray!60] (x23) -- (x14);
\draw[->,color=gray!60] (x14) -- (x05);
\draw[->,color=gray!60] (x50) -- (x40);
\draw[->,color=gray!60] (x40) -- (x30);
\draw[->,color=gray!60] (x30) -- (x20);
\draw[->,color=gray!60] (x20) -- (x10);
\draw[->,color=gray!60] (x10) -- (x00);
\draw[->,color=gray!60] (x00) -- (x05);
%
\begin{scope}[shift={(60:5)},xscale=1,yscale=-1]
\node (x00) at (0,0) [point] {}; \node[fill] (x10) at (1,0) [point] {}; \node (x20) at (2,0) [point] {}; \node[fill] (x30) at (3,0) [point] {};
\node (x40) at (4,0) [point] {}; \node[fill] (x50) at (5,0) [point] {};
\begin{scope}[shift={(60:1)}]
\node (x01) at (0,0) [point] {}; \node[fill] (x11) at (1,0) [point] {}; \node (x21) at (2,0) [point] {}; \node[fill] (x31) at (3,0) [point] {};
\node (x41) at (4,0) [point] {};
\begin{scope}[shift={(2,0)}]
\node (a2) at (-60:0.2) {$*$};
\node (a3) at (-120:0.2) {$*$};
\end{scope}
 \end{scope}
\begin{scope}[shift={(60:2)}]
\node (x02) at (0,0) [point] {}; \node[fill] (x12) at (1,0) [point]
{}; \node (x22) at (2,0) [point] {}; \node[fill] (x32) at (3,0)
[point] {};
\begin{scope}[shift={(1,0)}]
\node (a4) at (180:0.2) {$*$};
\node (a5) at (120:0.2) {$*$};
\node (a5b) at (120:0.1) {$*$};
\end{scope}
\begin{scope}[shift={(2,0)}]
\node (a1c) at (0:0.3) {$*$};
\node (a6) at (60:0.2) {$*$};
\end{scope}
 \end{scope}
\begin{scope}[shift={(60:3)}]
\node (x03) at (0,0) [point] {}; \node[fill] (x13) at (1,0) [point]
{}; \node (x23b) at (2,0) {};
 \end{scope}
\begin{scope}[shift={(60:4)}]
\node (x04) at (0,0) [point] {}; \node (x14) at (1,0) [point] {};
 \end{scope}
\begin{scope}[shift={(60:5)}]
\node (x05) at (0,0) [point] {};
 \end{scope}
\draw[->] (a1c) to (a2); \draw[->] (a2) to (a3);
\draw[->] (a3) to (a4); \draw[->] (a4) to (a5b);
\draw[->] (a5) to (a6); \draw[->] (a6) to (a1c);
\node at ($(a6)+(0.03,-0.14)$) {\rotatebox{40}{\dots}};
\draw[->] (a5b) to ($(a6)+(-0.1,-0.08)$);
\draw[->,color=gray!60] (x50) -- (x41);
\draw[->,color=gray!60] (x41) -- (x32);
\draw[->,color=gray!60] (x32) -- (x23b);
\draw[->,color=gray!60] (x23b) -- (x14);
\draw[->,color=gray!60] (x14) -- (x05);
\draw[->,color=gray!60] (x50) -- (x40);
\draw[->,color=gray!60] (x40) -- (x30);
\draw[->,color=gray!60] (x30) -- (x20);
\draw[->,color=gray!60] (x20) -- (x10);
\draw[->,color=gray!60] (x10) -- (x00);
\draw[->,color=gray!60] (x00) -- (x05);
\end{scope}
\node[fill=white] (n0) at (x23) { $\un{77(v^7)}$};
\draw[->] (n0) to[out=0,in=180] (a5);
\draw[->] (n0) to[out=-120,in=60] (a1a);
 \end{tikzpicture}
}
\end{gather*}
In either alcove the sequence beginning at $77\big(v^7\big)$ is
\begin{gather*}
79\big(v^7\big), 81\big(v^7\big), 84\big(v^8\big), 86\big(v^8\big), 89\big(v^7\big), 91\big(v^7\big), \dots\,.
\end{gather*}
\item Finally, an example with $\ell = 3$ (again we display both alcoves):
\begin{gather*}
\resizebox{10cm}{!}{
\begin{tikzpicture}
[xscale=3,yscale=3, point/.style={inner sep=0pt,circle,fill,minimum size=1mm,color=gray}]
\node (x00) at (0,0) [point] {}; \node[fill] (x10) at (1,0) [point] {}; \node (x20) at (2,0) [point] {}; \node[fill] (x30) at (3,0) [point] {};
\begin{scope}[shift={(60:1)}]
\node (x01) at (0,0) [point] {}; \node[fill] (x11) at (1,0) [point] {}; \node (x21) at (2,0) [point] {}; \node[fill] (x31) at (3,0) [point] {};
 \end{scope}
\begin{scope}[shift={(60:2)}]
\node (x02) at (0,0) [point] {}; \node (x12) at (1,0) {}; \node (x22) at (2,0) [point] {}; \node[fill] (x32) at (3,0) [point] {};
 \end{scope}
\begin{scope}[shift={(60:3)}]
\node (x03) at (0,0) [point] {}; \node (x13) at (1,0) [point] {}; \node (x23) at (2,0) [point] {}; \node[fill] (x33) at (3,0) [point] {};
 \end{scope}
\draw[color=gray!70] (x00) -- (x01) -- (x02) -- (x03) -- (x13) --
(x23) -- (x33) -- (x32) -- (x31) -- (x30) -- (x20) -- (x10) -- (x00);
\draw[color=gray!70] (x03) -- (x12) -- (x21) -- (x30);
 \node[fill=white] (n0) at ($(x12)+(0,0)$) {\tiny
 $\un{13(v)}$};
\begin{scope}[shift={(60:1)}] \begin{scope}[shift={(1,0)}]
\node (n1) at (150:0.3) {$*$};
\node (n2) at (-90:0.25) {$*$};
\node (n3) at (30:0.2) {$*$};
\node (n4) at (150:0.15) {$*$};
\node (n5) at (-90:0.1) {$\dots$};
 \end{scope} \end{scope}
%
\draw[->] (n0) to (n1);
\draw[->] (n1) to (n2);
\draw[->] (n2) to (n3);
\draw[->] (n3) to (n4);
\draw[->] (n4) to (n5);
\begin{scope}[shift={(60:2)}] \begin{scope}[shift={(2,0)}]
\node (p1) at (90:0.3) {$*$};
\node (p2) at (-30:0.25) {$*$};
\node (p3) at (-150:0.2) {$*$};
\node (p4) at (90:0.15) {$*$};
\node (p5) at (-30:0.1) {\rotatebox{30}{$\dots$}};
 \end{scope} \end{scope}
\draw[->] (n0) to (p1);
\draw[->] (p1) to (p2);
\draw[->] (p2) to (p3);
\draw[->] (p3) to (p4);
\draw[->] (p4) to ($(p5)+(0,0.05)$);
 \end{tikzpicture}
}
\end{gather*}
The sequence at either middle vertex is: $15(v)$, $18(v^2)$, $21(v)$, $24(v^2)$, $27(v)$, $\dots $.
\end{enumerate}
\end{Example}

Now, with $q = (\mu, m)$ as above (so that $\mu$ is neither a corner or almost corner) we apply the above algorithm to both alcoves which contain $\mu$ in their closure to produce multisets $\Lambda'$ and $\Lambda''$. We then remove $q$ from both, and def\/ine $\Lambda_q$ to be the union of the resulting multisets. (Thus~$\Lambda_q$ contains inf\/initely many points from both alcoves, but does not contain $q$ itself.)

The above algorithm produces, for each each seed $q \in Y$, a multiset $\Lambda_q$. We def\/ine $Z$ to be the union
\begin{gather*}
Z := Y \cup \bigcup_{\text{seeds } q \in Y} \Lambda_q.
\end{gather*}
Finally, we def\/ine $\widetilde{Z} := Z \setminus X$.

\begin{Example} \label{ex:Z} In Fig.~1 we display $\widetilde{Z}$ for $\ell = 5$ and all labelled points $\big(\mu, n\big(v^k\big)\big)$ with $n \le 82$. In the notation of Section~\ref{sec:conj}, the labelled points corresponding to $\mu \in \Chi_+$ are displayed in the smaller of the two alcoves contained in~$B_\mu$. The reader is referred to~\cite{WData} for further examples.
\end{Example}

\begin{Remark} For all $\big(\mu,n\big(v^k\big)\big) \in \widetilde{Z}$ it is easy to see that $\mu \in \Chi_{++} := \ZM_{>0} \varpi_1 \oplus \ZM_{>0} \varpi_2$.
\end{Remark}

\section{An altenative construction}

In this section we outline an alternative construction of the multiset $\widetilde{Z}$, which is more natural in some respects.

In the previous section we regarded the dominant weights $\Chi_+$ as the vertices of a~direc\-ted graph. An important dif\/ference in the current
construction is that now we consider \emph{all} weights~$\Chi$. That is, we consider the elements of~$\Chi$ as the vertices of a directed graph with edges $\l \to \l + \g$ if $\l, \l + \g \in \Chi_+$ and $\g \in \{ \varpi_1, \varpi_2 - \varpi_1, -\varpi_2 \}$. We write $\mu\to \mu'$ to indicate that there is a directed edge from~$\mu$ to~$\mu'$.

We use similar terminology to earlier: the notions of a \emph{wall point}, \emph{corner point}, \emph{almost corner} and \emph{$\ell$-regular point} extend in an obvious way to $\Chi$. Note that almost corners are necessarily wall points.

\subsection{Trees}
For us a \emph{rooted tree} is what many call an ``arborescence'': a directed graph with a distinguished vertex (the \emph{source}) such that there is a unique directed path from the source to any other vertex.\footnote{This is equivalent in an obvious way to the usual notion of
rooted tree.}

Given two rooted trees $A$ and $B$, denote by $A* B$ the rooted tree obtained by adjoining one copy of $B$ to each sink in $A$ at the source of $B$. That is, if $t_1, \dots, t_k$ denote the sinks of $A$ and we denote by $b_1, \dots, b_k$ the sources in $B \sqcup B \sqcup \dots \sqcup B$ ($k$~factors), then
\begin{gather*}
A * B := ( A \sqcup B \sqcup B \sqcup \cdots \sqcup B) / ( t_i \sim b_i).
\end{gather*}
 This operation is associative. The image of $A$ (resp. a copy of $B$) under this quotient map will be called a \emph{component} of type~$A$ (resp.~$B$).

Recall that $\ell \ge 3$ is f\/ixed. Consider the two rooted trees:
\begin{gather*}
\begin{tikzpicture}[point/.style={inner sep=0pt,circle,fill,minimum size=1mm,color=black}]
\node at (-1, 1.7) {$I := $};
\node[point] (a0) at (0,1.7) {};
\node[point] (a1) at (1, 1.7) {};
\node[point] (a2) at (2, 1.7) {};
\node (a3) at (3, 1.7) {\dots};
\node[point] (a4) at (4, 1.7) {};
\node[point] (a5) at (5, 1.7) {};
\draw[->] (a0) -- (a1); \draw[->] (a1)-- (a2);
\draw[->] (a2) -- (a3); \draw[->] (a3) -- (a4); \draw[->] (a4) --
(a5);
\node[point] (a0) at (0,0) {};
\node[point] (a1) at (1,0.4) {};
\node[point] (a1p) at (1,-0.4) {};
\node[point] (a2) at (2, 0.4) {};
\node[point] (a2p) at (2, -0.4) {};
\node (a3) at (3,0.4) {\dots};
\node (a3p) at (3,-0.4) {\dots};
\node[point] (a4) at (4, 0.4) {};
\node[point] (a4p) at (4, -0.4) {};
\node[point] (a5) at (5, 0.4) {};
\node[point] (a5p) at (5, -0.4) {};
\node[point] (a6) at (6, 0.4) {};
\node[point] (a6p) at (6, -0.4) {};
\node (a7) at (7, 0.4) {\dots};
\node (a7p) at (7, -0.4) {\dots};
\draw[->] (a0) -- (a1); \draw[->] (a1)-- (a2); \draw[->] (a2) -- (a3);
\draw[->] (a3) -- (a4); \draw[->] (a4) --(a5); \draw[->] (a5) --(a6);
\draw[->] (a6) --(a7);
\draw[->] (a0) -- (a1p); \draw[->] (a1p)-- (a2p); \draw[->] (a2p) -- (a3p);
\draw[->] (a3p) -- (a4p); \draw[->] (a4p) --(a5p); \draw[->] (a5p) --(a6p); \draw[->] (a6p) --(a7p);
\node at (-1,0) {$J :=$};
\node[point] (a5u) at (5,1) {};
\node[point] (a5up) at (5,-1) {};
\draw[->] (a4) -- (a5u); \draw[->] (a4p) -- (a5up);
\node at (0,-1.5) {{\tiny 0}};
\node at (1,-1.5) {{\tiny 1}};
\node at (2,-1.5) {{\tiny 2}};
\node at (4,-1.5) {{\tiny $\ell-1$}};
\node at (5,-1.5) {{\tiny $\ell$}};
\node at (6,-1.5) {{\tiny $\ell+1$}};
\end{tikzpicture}
\end{gather*}
The integers indicate the distance from the unique
source in each graph.

\begin{Example} \label{eg:IJJ}
We illustrate the operation $*$ with our graphs $I$ and $J$. For $\ell = 3$, the graph $I * J * J$ looks as follows:
\begin{gather*}
 \begin{tikzpicture}[scale=0.7, point/.style={inner
 sep=0pt,circle,fill,minimum size=1mm},
wallpoint/.style={inner sep=0pt,circle,draw,minimum size=2mm}]
\begin{scope}[shift={(7,4)}]
\foreach \x in {0,1,...,5} \node[point] (x4\x) at (\x,0) {};
\foreach \x in {0,1,...,4} \draw[->] (x4\x) to +(0.9,0);
\node at (6,0) {$\dots$}; \draw[->] (x45) to +(0.6,0);
\node[point] (x4out) at (2,1) {};
\draw[->] (x41) to[out=90,in=180] (x4out);
\end{scope}
\begin{scope}[shift={(7,3)}]
\foreach \x in {0,1,...,5} \node[point] (x3\x) at (\x,0) {};
\foreach \x in {0,1,...,4} \draw[->] (x3\x) to +(0.9,0);
\node at (6,0) {$\dots$}; \draw[->] (x35) to +(0.6,0);
\node[point] (x3out) at (2,-1) {};
\draw[->] (x31) to[out=-90,in=180] (x3out);
\end{scope}
\begin{scope}[shift={(7,-2)}]
\foreach \x in {0,1,...,5} \node[point] (xm3\x) at (\x,0) {};
\foreach \x in {0,1,...,4} \draw[->] (xm3\x) to +(0.9,0);
\node at (6,0) {$\dots$}; \draw[->] (xm35) to +(0.6,0);
\node[point] (x3mout) at (2,1) {};
\draw[->] (xm31) to[out=90,in=180] (x3mout);
\end{scope}
\begin{scope}[shift={(7,-3)}]
\foreach \x in {0,1,...,5} \node[point] (xm4\x) at (\x,0) {};
\foreach \x in {0,1,...,4} \draw[->] (xm4\x) to +(0.9,0);
\node at (6,0) {$\dots$}; \draw[->] (xm45) to +(0.6,0);
\node[point] (x4mout) at (2,-1) {};
\draw[->] (xm41) to[out=-90,in=180] (x4mout);
\end{scope}
\node[point] (su) at (6,3.5) {};
\node[point] (sd) at (6,-2.5) {};
\draw[->] (su) -- (x40); \draw[->] (su) -- (x30);
\draw[->] (sd) -- (xm40); \draw[->] (sd) -- (xm30);
\begin{scope}[shift={(4,1)}]
\foreach \x in {0,1,...,8} \node[point] (x1\x) at (\x,0) {};
\foreach \x in {0,1,...,7} \draw[->] (x1\x) to +(0.9,0);
\node at (9,0) {$\dots$}; \draw[->] (x18) to +(0.6,0);
\draw[->] (x11) to[out=90,in=180] (su);
\end{scope}
\begin{scope}[shift={(4,0)}]
\foreach \x in {0,1,...,8} \node[point] (xm1\x) at (\x,0) {};
\foreach \x in {0,1,...,7} \draw[->] (xm1\x) to +(0.9,0);
\node at (9,0) {$\dots$}; \draw[->] (xm18) to +(0.6,0);
\draw[->] (xm11) to[out=-90,in=180] (sd);
\end{scope}
\node[point] (x00) at (0,0.5) {};
\foreach \x in {1,2,3} \node[point] (x0\x) at (\x,0.5) {};
\foreach \x in {0,1,...,2} \draw[->] (x0\x) to +(0.9,0);
\draw[->] (x03) to (x10); \draw[->] (x03) to (xm10);
 \end{tikzpicture}
\end{gather*}
\end{Example}

Now def\/ine:
\begin{gather*}
J_\infty := \big( I * J^{* \ell - 1}\big)^{* \infty} := \big( I * J^{* \ell - 1}\big) * \big(I * J^{* \ell - 1}\big) * \big( I * J^{* \ell - 1}\big) * \cdots.
\end{gather*}
Recall the notion of component from above. We say that $v \in J_\infty$ is of \emph{wall type} if it lies in a component of type $I$ or is the source in a~component of type $J$. We say that $v$ is of \emph{almost corner type} if it is the source in a component of type $I$. (Thus if $v \in
J_\infty$ is of almost corner type then it is also of wall type.)

\begin{Example} We continue Example \ref{eg:IJJ} with $\ell = 3$. The points of wall type (resp. of almost corner type) in $J_\infty$
are depicted as open (resp.\ f\/illed) circles:
\begin{gather*}
 \begin{tikzpicture}[scale=0.7, point/.style={inner
 sep=0pt,circle,fill,minimum size=1mm},
wallpoint/.style={inner sep=0pt,circle,draw,minimum size=2mm}]
\begin{scope}[shift={(9,5)}]
\node[wallpoint,fill] (x50) at (0,0) {};
\foreach \x in {1,2,3} \node[wallpoint] (x5\x) at (\x,0) {};
\foreach \x in {0,1,2} \draw[->] (x5\x) to +(0.9,0);
\node (x5u) at (4,0.5) {\dots};
\node (x5d) at (4,-0.5) {\dots};
\draw[->] (x53) to (x5u); \draw[->] (x53) to (x5d);
\end{scope}
\begin{scope}[shift={(7,4)}]
\foreach \x in {0,1,...,5} \node[point] (x4\x) at (\x,0) {};
\foreach \x in {0,1,...,4} \draw[->] (x4\x) to +(0.9,0);
\node at (6,0) {$\dots$}; \draw[->] (x45) to +(0.6,0);
\draw[->] (x41) to[out=90,in=180] (x50);
\end{scope}
\begin{scope}[shift={(9,2)}]
\node[wallpoint,fill] (x20) at (0,0) {};
\foreach \x in {1,2,3} \node[wallpoint] (x2\x) at (\x,0) {};
\foreach \x in {0,1,2} \draw[->] (x2\x) to +(0.9,0);
\node (x2u) at (4,0.5) {\dots};
\node (x2d) at (4,-0.5) {\dots};
\draw[->] (x23) to (x2u); \draw[->] (x23) to (x2d);
\end{scope}
\begin{scope}[shift={(7,3)}]
\foreach \x in {0,1,...,5} \node[point] (x3\x) at (\x,0) {};
\foreach \x in {0,1,...,4} \draw[->] (x3\x) to +(0.9,0);
\node at (6,0) {$\dots$}; \draw[->] (x35) to +(0.6,0);
\node[wallpoint] (x3out) at (2,-1) {};
\draw[->] (x31) to[out=-90,in=180] (x3out);
\end{scope}
\begin{scope}[shift={(9,-1)}]
\node[wallpoint,fill] (x50) at (0,0) {};
\foreach \x in {1,2,3} \node[wallpoint] (x5\x) at (\x,0) {};
\foreach \x in {0,1,2} \draw[->] (x5\x) to +(0.9,0);
\node (x5u) at (4,0.5) {\dots};
\node (x5d) at (4,-0.5) {\dots};
\draw[->] (x53) to (x5u); \draw[->] (x53) to (x5d);
\end{scope}
\begin{scope}[shift={(7,-2)}]
\foreach \x in {0,1,...,5} \node[point] (xm3\x) at (\x,0) {};
\foreach \x in {0,1,...,4} \draw[->] (xm3\x) to +(0.9,0);
\node at (6,0) {$\dots$}; \draw[->] (xm35) to +(0.6,0);
\node[wallpoint] (x3mout) at (2,1) {};
\draw[->] (xm31) to[out=90,in=180] (x3mout);
\end{scope}
\begin{scope}[shift={(9,-4)}]
\node[wallpoint,fill] (x50) at (0,0) {};
\foreach \x in {1,2,3} \node[wallpoint] (x5\x) at (\x,0) {};
\foreach \x in {0,1,2} \draw[->] (x5\x) to +(0.9,0);
\node (x5u) at (4,0.5) {\dots};
\node (x5d) at (4,-0.5) {\dots};
\draw[->] (x53) to (x5u); \draw[->] (x53) to (x5d);
\end{scope}
\begin{scope}[shift={(7,-3)}]
\foreach \x in {0,1,...,5} \node[point] (xm4\x) at (\x,0) {};
\foreach \x in {0,1,...,4} \draw[->] (xm4\x) to +(0.9,0);
\node at (6,0) {$\dots$}; \draw[->] (xm45) to +(0.6,0);
\node[wallpoint] (x4mout) at (2,-1) {};
\draw[->] (xm41) to[out=-90,in=180] (x4mout);
\end{scope}
\node[wallpoint] (su) at (6,3.5) {};
\node[wallpoint] (sd) at (6,-2.5) {};
\draw[->] (su) -- (x40); \draw[->] (su) -- (x30);
\draw[->] (sd) -- (xm40); \draw[->] (sd) -- (xm30);
\begin{scope}[shift={(4,1)}]
\foreach \x in {0,1,...,8} \node[point] (x1\x) at (\x,0) {};
\foreach \x in {0,1,...,7} \draw[->] (x1\x) to +(0.9,0);
\node at (9,0) {$\dots$}; \draw[->] (x18) to +(0.6,0);
\draw[->] (x11) to[out=90,in=180] (su);
\end{scope}
\begin{scope}[shift={(4,0)}]
\foreach \x in {0,1,...,8} \node[point] (xm1\x) at (\x,0) {};
\foreach \x in {0,1,...,7} \draw[->] (xm1\x) to +(0.9,0);
\node at (9,0) {$\dots$}; \draw[->] (xm18) to +(0.6,0);
\draw[->] (xm11) to[out=-90,in=180] (sd);
\end{scope}
\node[fill,wallpoint] (x00) at (0,0.5) {};
\foreach \x in {1,2,3} \node[wallpoint] (x0\x) at (\x,0.5) {};
\foreach \x in {0,1,...,2} \draw[->] (x0\x) to +(0.9,0);
\draw[->] (x03) to (x10); \draw[->] (x03) to (xm10);
 \end{tikzpicture}
\end{gather*}
\end{Example}

\subsection{The map} Fix an almost corner $\l$. We will construct a map
\begin{gather*}
\Phi = \Phi_\l \colon \ J_\infty \to \Chi.
\end{gather*}
It will have the following two properties:
\begin{enumerate}\itemsep=0pt
\item If $v \to v'$ in $J_\infty$ then either $\Phi(v) = \Phi(v')$ or $\Phi(v) \to \Phi(v')$ in $\Chi$.
\item For $v \in J_\infty$, $\Phi(v)$ is a wall point (resp.\ an almost corner) if and only if $v$ is of wall type (resp.\ of almost corner type).
\end{enumerate}

Because $J_\infty$ is the union of components of types $I$ or $J$ it will suf\/f\/ice to def\/ine the image of~$\Phi$ on the source of $J_\infty$, and then def\/ine it on each component of type $I$ or $J$ inductively. This is what we do now:

\emph{Rule $0$:} Let $v_0$ denote the source of $J_\infty$. Set $\Phi(v_0) = \l$ (our f\/ixed almost corner).

\emph{Rule $1$:} Let $I' \subset J_\infty$ denote a component of type $I$ and let $v_0'$ denote the source of $I'$. Suppose that $\mu :=\Phi(v_0')$ is def\/ined, but that $\Phi$ is not def\/ined on the rest of $I'$. We def\/ine $\Phi$ on $I'$ as follows:
\begin{gather*}
\mu_1 \to \mu_1 \to \mu_2 \to \dots \to \mu_{\ell-2} \to \mu_{\ell-1} \to
\mu_{\ell-1}.
\end{gather*}
Here $\mu = \mu_1, \mu_2, \dots, \mu_\ell$ are uniquely determined by the requirement that $\mu = \mu_1$ and $\mu_1 \to \mu_2 \to \cdots \to
\mu_{\ell-1}$ are all wall points. (The fact that $\mu_1, \dots,\mu_{\ell-1}$ are well def\/ined is a consequence of property~(2) above: $\Phi(v_0')$ is an almost corner.)

\emph{Rule $2$:} Let $J' \subset J_\infty$ denote a component of type $J$ and let $v_0'$ denote the source of $J'$. Suppose that $\mu :=\Phi(v_0')$ is def\/ined, but that $\Phi$ is not def\/ined on all of $J'$.

By property (2) above $\mu$ is a wall-point but is not an almost corner. Hence there exist two edges $d', d''$ with source $\mu$ and image an $\ell$-regular point. For $d \in \{ d' , d'' \}$, def\/ine a sequence $\{ (\mu_i,d_i) \}_{i \ge 0}$ as follows:
\begin{enumerate}\itemsep=0pt
\item[1)] $(\mu_0, d_0) := (\mu, d)$;
\item[2)] if $\mu_i + d_i$ is an interior point, def\/ine $(\mu_{i+1}, d_{i+1}) := (\mu_i + d_i, d_i)$;
\item[3)] if $\mu_i + d_i$ is a wall point, def\/ine $(\mu_{i+1}, d_{i+1}) := (\mu_i, r(d_i))$, where $r$ denotes the ref\/lection in the
 (unique) wall containing $\mu_i + d_i$.
\end{enumerate}
Denote by $\{ (\mu'_i,d'_i) \}_{i \ge 0}$ (resp.\ $\{ (\mu''_i,d''_i)) \}_{i \ge 0})$ the sequences associated to $d = d'$ (resp.\ $d = d''$).

We def\/ine $\Phi$ on $J'$ as follows:
\begin{gather*}
\begin{tikzpicture}[scale=1.6]
\node (a0) at (0,0) {$\mu'_0 = \mu''_0$};
\node (a1) at (1,0.4) {$\mu'_1$};
\node (a1p) at (1,-0.4) {$\mu''_1$};
\node (a2) at (2, 0.4) {$\mu'_2 $};
\node (a2p) at (2, -0.4) {$\mu''_2$};
\node (a3) at (3,0.4) {\dots};
\node (a3p) at (3,-0.4) {\dots};
\node (a4) at (4, 0.4) {$\mu'_{\ell-1}$};
\node (a4p) at (4, -0.4) {$\mu''_{\ell-1}$};
\node (a5) at (5, 0.4) {$\mu'_\ell$};
\node (a5p) at (5, -0.4) {$\mu''_{\ell}$};
\node (a6) at (6, 0.4) {$\mu'_{\ell+1}$};
\node (a6p) at (6, -0.4) {$\mu''_{\ell + 1}$};
\node (a7) at (7, 0.4) {\dots};
\node (a7p) at (7, -0.4) {\dots};
\draw[->] (a0) -- (a1); \draw[->] (a1)-- (a2); \draw[->] (a2) -- (a3);
\draw[->] (a3) -- (a4); \draw[->] (a4) --(a5); \draw[->] (a5) --(a6);
\draw[->] (a6) --(a7);
\draw[->] (a0) -- (a1p); \draw[->] (a1p)-- (a2p); \draw[->] (a2p) -- (a3p);
\draw[->] (a3p) -- (a4p); \draw[->] (a4p) --(a5p); \draw[->] (a5p) --(a6p); \draw[->] (a6p) --(a7p);
\node (a5u) at (5,1) {$\mu'_{\ell-1} + d'_{\ell-1}$};
\node (a5up) at (5,-1) {$\mu''_{\ell-1} + d''_{\ell-1}$};
\draw[->] (a4) -- (a5u); \draw[->] (a4p) -- (a5up);
\end{tikzpicture}
\end{gather*}
(Note that $\mu'_{\ell-1} + d'_{\ell-1}$ and $\mu''_{\ell-1} + d''_{\ell-1}$ are wall points.)

In order for the above def\/inition to be well def\/ined the we should check that properties~(1) and~(2) are satisf\/ied at each step. properties~(1) and~(2) for wall points are immediate from the def\/initions. Property (2) for almost corner points follows from the following observation: suppose that $J' \subset J_\infty$ is a component of type $J$, let $v_0'$ denote its source, and let $v_1'$, $v_1''$ denote the two sinks (so that $v_1'$ and $v_1''$ are of wall type in $J_\infty$). Then if $\Phi(v_0')$ is of distance $k \le \ell - 1$ from a corner point, then $\Phi(v_1')$ and $\Phi(v_1'')$ are of distance $k-1$ from a corner point.

\begin{Example}
We illustrate $\Phi$ on a~component of type $J$ with $\ell = 5$ (the source is marked with a circle):
\begin{gather*}
\resizebox{14cm}{!}{
\begin{tikzpicture}
[xscale=3.5,yscale=3.5, point/.style={inner sep=0pt,circle,fill,minimum size=1mm,color=gray}]
\node (x00) at (0,0) [point] {}; \node[fill] (x10) at (1,0) [point] {}; \node (x20) at (2,0) [point] {}; \node[fill] (x30) at (3,0) [point] {};
\node (x40) at (4,0) [point] {}; \node[fill] (x50) at (5,0) [point] {};
\begin{scope}[shift={(60:1)}]
\node (x01) at (0,0) [point] {}; \node[fill] (x11) at (1,0) [point] {}; \node (x21) at (2,0) [point] {}; \node[fill] (x31) at (3,0) [point] {};
\node (x41) at (4,0) [point] {};
\begin{scope}[shift={(2,0)}]
\node (a2) at (-60:0.2) {$*$};
\node (a2b) at (-90:0.075) {$\rotatebox{0}{\dots}$};
\node (a3) at (-120:0.2) {$*$};
\end{scope}
 \end{scope}
\begin{scope}[shift={(60:2)}]
\node (x02) at (0,0) [point] {}; \node[fill] (x12) at (1,0) [point]
{}; \node (x22) at (2,0) [point] {}; \node[fill] (x32) at (3,0)
[point] {};
\begin{scope}[shift={(1,0)}]
\node (a4) at (180:0.2) {$*$};
\node (a5) at (120:0.2) {$*$};
\end{scope}
\begin{scope}[shift={(2,0)}]
\node (a1b) at (0:0.2) {$*$};
\node (a1a) at (0:0.3) {$*$};
\node (a6) at (60:0.2) {$*$};
\end{scope}
 \end{scope}
\begin{scope}[shift={(60:3)}]
\node (x03) at (0,0) {*}; \node[fill] (x13) at (1,0) [point]
{}; \node (x23) at (2,0) {};
 \end{scope}
\begin{scope}[shift={(60:4)}]
\node (x04) at (0,0) [point] {}; \node (x14) at (1,0) [point] {};
 \end{scope}
\begin{scope}[shift={(60:5)}]
\node (x05) at (0,0) [point] {};
 \end{scope}
\draw[->] (a1a) to (a2); \draw[->] (a2) to (a3);
\draw[->] (a3) to (a4); \draw[->] (a4) to (a5);
\draw[->] (a5) to (a6); \draw[->] (a6) to (a1b);
\draw[->] (a1b) to ($(a2b)+(0.09,0.05)$);
\draw[->,color=gray!60] (x50) -- (x41);
\draw[->,color=gray!60] (x41) -- (x32);
\draw[->,color=gray!60] (x32) -- (x23);
\draw[->,color=gray!60] (x23) -- (x14);
\draw[->,color=gray!60] (x14) -- (x05);
\draw[->,color=gray!60] (x50) -- (x40);
\draw[->,color=gray!60] (x40) -- (x30);
\draw[->,color=gray!60] (x30) -- (x20);
\draw[->,color=gray!60] (x20) -- (x10);
\draw[->,color=gray!60] (x10) -- (x00);
\draw[->,color=gray!60] (x00) -- (x05);
\draw[->] (a4) -- (x03);
%
\begin{scope}[shift={(60:5)},xscale=1,yscale=-1]
\node (x00) at (0,0) [point] {}; \node[fill] (x10) at (1,0) [point] {}; \node (x20) at (2,0) {*}; \node[fill] (x30) at (3,0) [point] {};
\node (x40) at (4,0) [point] {}; \node[fill] (x50) at (5,0) [point] {};
\begin{scope}[shift={(60:1)}]
\node (x01) at (0,0) [point] {}; \node[fill] (x11) at (1,0) [point] {}; \node (x21) at (2,0) [point] {}; \node[fill] (x31) at (3,0) [point] {};
\node (x41) at (4,0) [point] {};
\begin{scope}[shift={(2,0)}]
\node (a2) at (-60:0.2) {$*$};
\node (a3) at (-120:0.2) {$*$};
\end{scope}
 \end{scope}
\begin{scope}[shift={(60:2)}]
\node (x02) at (0,0) [point] {}; \node[fill] (x12) at (1,0) [point]
{}; \node (x22) at (2,0) [point] {}; \node[fill] (x32) at (3,0)
[point] {};
\begin{scope}[shift={(1,0)}]
\node (a4) at (180:0.2) {$*$};
\node (a5) at (120:0.2) {$*$};
\node (a5b) at (120:0.1) {$*$};
\end{scope}
\begin{scope}[shift={(2,0)}]
\node (a1c) at (0:0.3) {$*$};
\node (a6) at (60:0.2) {$*$};
\end{scope}
 \end{scope}
\begin{scope}[shift={(60:3)}]
\node (x03) at (0,0) [point] {}; \node[fill] (x13) at (1,0) [point]
{}; \node (x23b) at (2,0) {};
 \end{scope}
\begin{scope}[shift={(60:4)}]
\node (x04) at (0,0) [point] {}; \node (x14) at (1,0) [point] {};
 \end{scope}
\begin{scope}[shift={(60:5)}]
\node (x05) at (0,0) [point] {};
 \end{scope}
\draw[->] (a1c) to (a2); \draw[->] (a2) to (a3);
\draw[->] (a3) to (a4); \draw[->] (a4) to (a5b);
\draw[->] (a5) to (a6); \draw[->] (a6) to (a1c);
\node at ($(a6)+(0.03,-0.14)$) {\rotatebox{40}{\dots}};
\draw[->] (a5b) to ($(a6)+(-0.1,-0.08)$);
\draw[->,color=gray!60] (x50) -- (x41);
\draw[->,color=gray!60] (x41) -- (x32);
\draw[->,color=gray!60] (x32) -- (x23b);
\draw[->,color=gray!60] (x23b) -- (x14);
\draw[->,color=gray!60] (x14) -- (x05);
\draw[->,color=gray!60] (x50) -- (x40);
\draw[->,color=gray!60] (x40) -- (x30);
\draw[->,color=gray!60] (x30) -- (x20);
\draw[->,color=gray!60] (x20) -- (x10);
\draw[->,color=gray!60] (x10) -- (x00);
\draw[->,color=gray!60] (x00) -- (x05);
\draw[->] (a2) -- (x20);
\end{scope}
\node[circle,draw,fill=white] (n0) at (x23) {};
\draw[->] (n0) to[out=0,in=180] (a5);
\draw[->] (n0) to[out=-120,in=60] (a1a);
 \end{tikzpicture}
}
\end{gather*}
Note the similarity to the previous section. The only dif\/ference is the arrow joining an $\ell$-regular point to a wall point.
\end{Example}

\subsection[Extending $\Phi$]{Extending $\boldsymbol{\Phi}$} Let $v_0$ denote the source of~$J_\infty$. For any choice of label $n\big(v^{k_0}\big)$, we inductively extend $\Phi$ to produce a map
\begin{gather*}
\widetilde{\Phi} = \widetilde{\Phi}_{\l, n(v^{k_0})}\colon \ J_\infty \to \Chi \times \MC
\end{gather*}
such that $\widetilde{\Phi}(v_0) = \big(\l, n\big(v^{k_0}\big)\big)$. We proceed as follows: Suppose that $v \to v'$ in~$J_\infty$ and that $\widetilde{\Phi}$ is def\/ined on $v$ but not on $v'$, and let $\widetilde{\Phi}(v) = \big(\Phi(v), m\big(v^k\big)\big)$. Then
$\widetilde{\Phi}(v') = \big(\Phi(v'), m'\big(v^{k'}\big)\big)$ with
\begin{gather*}
m' := \begin{cases} m + 2 & \text{if $\Phi(v) \ne \Phi(v')$,} \\ m + 3 &
 \text{if $\Phi(v) = \Phi(v')$} \end{cases}
\end{gather*}
and
\begin{gather*}
k' := \begin{cases} k & \text{if $\Phi(v) \ne \Phi(v')$,}
\\ k + 1 & \begin{array}{@{}l@{}} \text{if $\Phi(v) = \Phi(v')$ and} \\\text{$v \to v'$ belongs to a
 component of type $I$}, \end{array}
\\ k \pm 1 & \begin{array}{@{}l@{}} \text{if $\Phi(v) = \Phi(v')$ and} \\\text{$v \to v'$ belongs to a
 component of type $J$}. \end{array}
\end{cases}
\end{gather*}
The sign ambiguity in the f\/inal case is resolved as follows: suppose that $J' \subset J_\infty$ is a component of type $J$, $v_0'$ is its initial vertex, and $\widetilde{\Phi}(v_0') = \big( \mu, m\big(v^k\big)\big)$. Then $\widetilde{\Phi}$ takes values in $\Chi \times \ZM_{\ge 0} \times \big\{ v^k, v^{k+1} \big\}$ on $J'$.

\subsection{Pruning}
Consider our map
\begin{gather*}
\Phi = \Phi_\l\colon \ J_\infty \to \Chi
\end{gather*}
(which depended on the choice of a f\/ixed almost corner $\l$). We now ``prune'' our tree $J_\infty$ to produce a new tree~$J_\l$. Consider the set
\begin{gather*}
K_\l := \left \{ v \in J_\infty \, \middle | \, \begin{array}{@{}c@{}} \text{there exists }v_1, \dots,
v_m\text{ such that } \Phi(v_1) \notin \Chi_{++} \\ \text{ and }v_1 \to
v_2 \to \dots \to v_m = v \end{array} \right \}
\end{gather*}
and def\/ine $J_\l$ to be the full subgraph with vertices $J_\infty
\setminus K_\l$. (That is, we remove all branches from $J_{\infty}$
that contain elements which are mapped to weights which are not
strictly dominant.) The restriction of
$\Phi$ to $J_\l$ def\/ines a map
\begin{gather*}
\Phi \colon \ J_{\l} \to \Chi_{++}.
\end{gather*}
Restricting the extension $\widetilde{\Phi}$ (which depended on an additional choice of label $n\big(v^k\big)$) of~$\Phi$ to~$J_{\infty,++}$ yields a map:
\begin{gather*}
\widetilde{\Phi} = \widetilde{\Phi}_{\l,n(v^k)}\colon \ J_\l \to \Chi_{++} \times \MC.
\end{gather*}

\subsection{The alternative construction} For any positive
integer $m$ let
\begin{gather*}
 \l_m := (\ell - 1)m\varpi_1 + \varpi_2 \qquad \text{and} \qquad g_m := \big(\l_m, (2m\ell - 1)\big(v^{-1}\big)\big).
\end{gather*}
The previous constructions provide us with a rooted tree $J_{\l_m}$ and a map
\begin{gather*}
 \widetilde{\Phi}_{g_m} \colon \ J_{\l_m} \to \Chi_{++} \times \MC.
\end{gather*}
Consider the union of multisets
\begin{gather*}
Z' := \bigcup_{m \ge 1} \big\{ \widetilde{\Phi}_{g_m}(v) \, | \, v \in J_{\l_m} \big\}.
\end{gather*}
The following lemma implies that the above algorithm provides an alternative construction of~$\widetilde{Z}$.
\begin{Lemma} We have $\widetilde{Z} = Z' \setminus \{ g_m \, | \, m \in \ZM_{\ge 0} \}$.
\end{Lemma}

\begin{proof}[Proof (sketch)]
Fix $m \!\ge\! 0$ and consider $\widetilde{\Phi} \!:= \!\widetilde{\Phi}_{g_m}$ as def\/ined above. Also, set $q {=} \big(\ell m \varpi_1, 0\big(v^0\big)\big)$ and consider the set $\Lambda_q$ def\/ined as above. We claim that we have
\begin{gather} \label{eq:eq}
\Lambda_q \setminus \{ q \} = \big\{ \widetilde{\Phi}_{g_m}(v) \, | \, v \in J_{\ell m \varpi_1}\big\} \setminus \{ g_m \},
\end{gather}
which implies the lemma. The equality \eqref{eq:eq} follows from the following local considerations:
\begin{enumerate}\itemsep=0pt
\item The restriction of $\widetilde{\Phi}$ to the initial component of type $I$ takes the same values as part (1) of ``Dynamics on walls'', except for the f\/irst two values ($g_m$ and $q$ respectively) which are removed in~\eqref{eq:eq}.
\item The restriction of $\widetilde{\Phi}$ to any component of type~$I$ other than the initial segment takes the same values as those produced by part (2) of ``Dynamics on the walls''.
\item If $v$ (resp.~$v'$,~$v''$) denotes the source (resp.\ sinks) of a~component of type $J$ then $\widetilde{\Phi}(v')$ and~$\widetilde{\Phi}(v'')$ are obtained from $\widetilde{\Phi}(v)$ by a~giant leap, i.e., part~(3) of ``Dynamics on the walls''.
\item For each component of type $J$ in $J_\infty$, if we consider the full subgraph
\begin{gather*}
\begin{tikzpicture}[point/.style={inner sep=0pt,circle,fill,minimum size=1mm,color=black}]
\node[point] (a0) at (0,0) {};
\node[point] (a1) at (1,0.4) {};
\node[point] (a1p) at (1,-0.4) {};
\node[point] (a2) at (2, 0.4) {};
\node[point] (a2p) at (2, -0.4) {};
\node (a3) at (3,0.4) {\dots};
\node (a3p) at (3,-0.4) {\dots};
\node[point] (a4) at (4, 0.4) {};
\node[point] (a4p) at (4, -0.4) {};
\node[point] (a5) at (5, 0.4) {};
\node[point] (a5p) at (5, -0.4) {};
\node[point] (a6) at (6, 0.4) {};
\node[point] (a6p) at (6, -0.4) {};
\node (a7) at (7, 0.4) {\dots};
\node (a7p) at (7, -0.4) {\dots};
\draw[->] (a0) -- (a1); \draw[->] (a1)-- (a2); \draw[->] (a2) -- (a3);
\draw[->] (a3) -- (a4); \draw[->] (a4) --(a5); \draw[->] (a5) --(a6);
\draw[->] (a6) --(a7);
\draw[->] (a0) -- (a1p); \draw[->] (a1p)-- (a2p); \draw[->] (a2p) -- (a3p);
\draw[->] (a3p) -- (a4p); \draw[->] (a4p) --(a5p); \draw[->] (a5p) --(a6p); \draw[->] (a6p) --(a7p);
\node at (-1,0) {$J_\int :=$};
\end{tikzpicture}
\end{gather*}
then the restriction of $\widetilde{\Phi}$ to each branch of $J_\int$ produces the same set as the algorithm in Section~\ref{sec:alcovebilliards} (``Billiards in an alcove''). \hfill $\qed$
\end{enumerate}\renewcommand{\qed}{}
\end{proof}

\section{The conjecture} \label{sec:conj}

Recall the notion of an $\ell$-alcove in $\Chi_\RM$ from above. An \emph{alcove} is a $1$-alcove. The alcove $\AC_0 := \{ \l \, | \, 0 <
\langle \alpha^\vee, \l \rangle < 1 \text{ for all } \alpha^\vee \in \Phi_+^\vee \}$ is the \emph{fundamental alcove}. The map $x
\mapsto x\AC_0$ gives a~bijection between the af\/f\/ine Weyl group $\Wa$ and the set of alcoves. It restricts to a bijection between $\f\Wa$ and the dominant alcoves.

Given $\mu \in \Chi$ the open box
\begin{gather*}
B_\mu := \big\{ \l \in \Chi_\RM \, | \, \langle \a^\vee , \mu \rangle < \langle \a^\vee , \l \rangle <
\langle \a^\vee , \mu \rangle + 1 \text{ for all $\a^\vee \in \Sigma^\vee$} \big\}
\end{gather*}
contains exactly 2 alcoves. In this way we obtain a map $\Wa \to \Chi$ by sending $x\in\Wa$ to the unique $\mu \in \Chi$ such that $B_\mu$ contains $x\AC_0$. It restricts to a map
\begin{gather*}
\kappa \colon \ \f\Wa \to \Chi_+.
\end{gather*}

Let $s_0$ denote the simple af\/f\/ine ref\/lection. Consider the elements
\begin{gather*}
x_0 := \id, \ x_1 = s_0, \ x_2 = s_0s_1, \ x_3 = s_0s_1s_2, \ x_4 := s_0s_1s_2s_0, \ \dots
\end{gather*}
of $\Wa$. (These are the alcoves along one edge of the dominant cone.) We have $\kappa(x_{2i}) = \kappa(x_{2i+1}) = i\varpi_1$ for all $i \ge 0$.

We now describe a sort of inverse to $\kappa$. Recall that $\Chi_{++} = \ZM_{>0} \varpi_1 \oplus \ZM_{>0} \varpi_2$ denotes the strictly dominant weights. Consider $\mu \in \Chi_{++}$ and let $x, x'$ denote the two elements of $\f\Wa$ indexing alcoves contained in $B_\mu$. If $\RC(z) = \{ s \in \Sa \, | \, zs < z \}$ denotes the right descent set then it is easy to see that we have
\begin{gather*}
\RC(x) \sqcup \RC(x') = \{ s_0, s_1, s_2 \} = \Sa \quad
\text{(disjoint union)}.
\end{gather*}
It follows that for any pair $\mu \in \Chi_{++}, s \in \Sa$ there is a unique element $x_\mu^s \in \f\Wa$ such that $s \in \RC(x_\mu^s)$ and $\mu =
\kappa(x_\mu^s)$.

Recall our prime $p$ from above, and consider the multiset $\widetilde{Z}$ constructed in the previous section with $\ell = p$. We now describe how to use $\widetilde{Z}$ to def\/ine new elements in the anti-spherical module $\AntiS_v$. Consider the $\ZM$-linear map $\varphi \colon \Zvv \to \Zv$
given by $v^0 \mapsto 1$ and $v^i \mapsto v^i + v^{-i}$ for $i > 0$. Set $\p\zeta_0 := \un{n}_{x_0}$. For any $i > 0$ let $s \in \Sa$ denote the unique element of $\RC(x_i)$ and consider the element:
\begin{gather*}
\p\zeta_i := \un{n}_{x_i} + \sum_{(\mu,n(v^k)) \in \widetilde{Z} ; \atop n \in \{i, i-1,i-2 \} } \varphi\big(v^k\big) \un{n}_{x_\mu^s}.
\end{gather*}

\begin{Conjecture} We have:
 \begin{enumerate}\itemsep=0pt
 \item[$1)$] $\p\zeta_i = \p\un{n}_{x_i}$ for $0 \le i < 2p(p+1)$;
 \item[$2)$] $\p\zeta_i = \p\un{n}^2_{x_i}$ for all $0 \le i$.
 \end{enumerate}
\end{Conjecture}

\begin{Remark} Some remarks on the conjecture:
 \begin{enumerate}\itemsep=0pt
 \item For an example of $\widetilde{Z}$ the reader is referred to Example~\ref{ex:Z}.
 \item It is a nice exercise to compare our conjecture to the results of J.G.~Jensen~\cite{JensenTilting} and Par\-ker~\cite{Parker}. For example, in Fig.~\ref{fig:p5}, the results of Jensen and Parker are explained by the unique alcoves labelled $12(1)$, $14(1)$, $16(1)$, $18(1)$ and~$21(v)$.
 \item Our conjecture (in particular the def\/inition of the set $\widetilde{Z}$) does not seem to make sense for $p = 2$.
 \item We have verif\/ied part (1) of the conjecture for $p = 3, 5$ and in many cases for $p = 7$ by computer. (In fact these calculations led to the conjecture.)
 \item To determine the $p$-canonical basis in $\AntiS_v$ it is enough to know the elements $\p\un{n}_{x_i}$ for all $i > 0$. (After exploiting the automorphism $s_0 \mapsto s_0$, $s_1 \mapsto s_2$, $s_2 \mapsto s_1$ this can be deduced from a~$v$-analogue of the fact that one can apply the tilting tensor product theorem to determine all tilting characters, provided one knows the tilting characters along the walls and in the $(p-1)\rho$-shift of the fundamental box, see \cite[Section~1.6]{WTakagi}. One can check by hand that one has
\begin{gather*}
\p\un{n}_x = \un{n}_x \qquad \text{for all $x \in \{ \id, s_0,
 s_0s_1s_2s_1, s_0s_1s_2s_1s_0 \}$}
\end{gather*}
and so the only remaining cases are $\p\un{n}_{x_i}$ for $i > 1$.)
\item Similarly, to know all elements $\p\un{n}^2_{x}$ for $x \in \f\Wa$ it should be enough to know $\p\un{n}^2_{x_i}$ for all $i$. Thus our conjecture gives a formula for the $\p\un{n}_x^2$ (which currently have no other rigorous def\/inition\footnote{However see the last sentence of Remark~\ref{rem:generations}.}).

\item Recall that the exists a bijection between two-sided cells in the af\/f\/ine Weyl group and nilpotent orbits in the dual group~\cite{Lorbits}. Moreover, every two sided cell intersects $\f \Wa$ in a left cell (the canonical left cell)~\cite{LX}. The ``dif\/f\/icult'' elements $x_i$ for $i > 0$ all lie in the left cell corresponding to the minimal nilpotent orbit of $\SL_3$. This suggests that the problem of determining tilting characters should be related to the geometry of nilpotent orbits. Related results (connecting the $p$-canonical basis to coherent sheaves on the Springer resolution) may be found in \cite{AR, ARFM}.
\item After taking Remark~\ref{rem:growth} into account, our conjecture implies that the multiplicities of Weyl modules in the tilting module $T_{ 3p k \varpi_1}$ grow (at least) exponentially in $k$. Thus our conjecture implies that decomposition numbers for symmetric groups $S_n$ grow exponentially in $n$ (see~\cite{DonkinTilting,Erdmann} for the connection between tilting module characters and decomposition numbers), and that the dimension of the tilting module $T_{k\varpi_1}$ grows exponentially in $k$. Neither of these statements is true for $\SL_2$ or for the quantum group of~$\SL_3$ at a $p^{\rm th}$ root of unity.
\item The powers of $v$ which occur after $n$ iterations of the algorithm of Section~\ref{sec:wallbilliards} are all at least~$n$. (By contrast, the algorithm Section~\ref{sec:alcovebilliards} only ever changes the power of~$v$ by~$\pm 1$.) This implies (assuming our conjecture) that $\p\un{n}_{x_i}$ involves arbitrarily high powers of~$v$ as~$i$ grows. This in turn implies that certain structure constants for the action of $\un{h}_{s_i}$ on the $p$-canonical basis involve arbitrarily high powers of $v$. One can use this observation (and \cite[Section~1.4]{RW}) to conclude that the analogue of the $a$-function for the $p$-canonical basis for the Hecke algebra $\He$ of $\Wa$ is unbounded.
\item Fix $m = \big(\mu,n\big(v^k\big)\big) \in \widetilde{Z}$ and let $x$, $x'$ index the two alcoves contained in $B_\mu$, chosen such that $\RC(x_n) \subset \RC(x)$. The formula for $\p\zeta_i$ above implies that $m$ contributes $\varphi\big(v^k\big)$ to the coef\/f\/icient of $\un{n}_{x}$ (resp.~$\un{n}_{x'}$, $\un{n}_{x}$) in $\p\zeta_n$ (resp.~$\p\zeta_{n+1}$, $\p\zeta_{n+2}$). A key step in arriving at the multiset $\widetilde{Z}$ is to observe that the elements of the $p$-canonical basis $\{ \p\un{n}_{x_i} \}$ may be decomposed into such ``triples''. After a f\/irst version of this paper was written, L.T.~Jensen~\cite{JensenStar} has proved that such a decomposition is always possible, as a consequence of more general results on the $p$-canonical basis and the star operations of Kazhdan and the f\/irst author.
 \end{enumerate}
\end{Remark}

Finally, let us explain how to go from the $p$-canonical basis to a~picture similar to that at the beginning of this paper (which is conjecturally described by $\widetilde{Z}$). It was this procedure (combined with heuristics as to what constitutes generation $2$) that led us to our conjecture.

Fix an alcove $y\AC_0$ (for $y \in \fW$) which is not on a wall. For all $i \ge 0$ such that $\p n_{y,x_i} \ne 0$, we consider the polynomial obtained from $\p n_{y,x_i}$ by discarding negative powers of $v$, and write~$i(f)$ in the alcove $y\AC_0$. This produces a diagram, in which the strictly dominant alcoves are decorated by symbols of the form $i(f)$.\footnote{For an example of such a diagram for $p = 5$ see \url{http://www.maths.usyd.edu.au/u/geordie/pCanA2/p5pretriples.pdf}.} Now, for all $\mu \in \Chi_{++}$ we replace each ``triple'' of the form
\begin{gather*}\begin{split}&
 \begin{tikzpicture}[yscale=-0.6, xscale=-0.6, point/.style={inner
 sep=0pt,circle,fill,minimum size=0mm}]
\node (x) at (0,0) [point] {};
\node (y) at (60:4) [point] {};
\node (u) at (4,0) [point] {};
\begin{scope}[shift={(60:4)}]
\node (z) at (4,0) [point] {};
\end{scope}
\draw (x) -- (y) -- (z);
\draw (x) -- (u) -- (z);
\draw (u) -- (y);
\node (a) at (35:5) {\tiny $i(f), (i+2)(f)$};
\node (b) at (30:2) {\tiny $(i+1)(f)$};
 \end{tikzpicture}\end{split}
\qquad \text{or} \qquad \begin{split}&
 \begin{tikzpicture}[yscale=-0.6, xscale=-0.6, point/.style={inner
 sep=0pt,circle,fill,minimum size=0mm}]
\node (x) at (0,0) [point] {};
\node (y) at (60:4) [point] {};
\node (u) at (4,0) [point] {};
\begin{scope}[shift={(60:4)}]
\node (z) at (4,0) [point] {};
\end{scope}
\draw (x) -- (y) -- (z);
\draw (x) -- (u) -- (z);
\draw (u) -- (y);
\node (a) at (20:1.9) {\tiny $i(f), (i+2)(f)$};
\node (b) at (30:5) {\tiny $(i+1)(f)$};
 \end{tikzpicture}\end{split}
\end{gather*}
by
\begin{gather*}
 \begin{tikzpicture}[yscale=-0.6, xscale=-0.6, point/.style={inner
 sep=0pt,circle,fill,minimum size=0mm}]
\node (x) at (0,0) [point] {};
\node (y) at (60:4) [point] {};
\node (u) at (4,0) [point] {};
\begin{scope}[shift={(60:4)}]
\node (z) at (4,0) [point] {};
\end{scope}
\draw (x) -- (y) -- (z);
\draw (x) -- (u) -- (z);
\draw (u) -- (y);
\node (a) at (30:5) {\tiny $i(f)$};
 \end{tikzpicture}
\end{gather*}
This is what is depicted (for a second generation version) in Fig.~\ref{fig:p5}. 

\subsection*{Acknowledgements}
We would like to thank the anonymous referees for their comments.

\pdfbookmark[1]{References}{ref}
\LastPageEnding

\end{document}